\documentclass[11pt]{article}

\usepackage{graphicx, amsmath, amssymb, amsthm}
\usepackage[margin=2cm]{geometry}
\usepackage{commath}               

\usepackage{setspace}
\usepackage{subfig}
\usepackage{array}
\usepackage{xparse}
\usepackage{color}
\usepackage{endnotes}
\usepackage{algorithm}
\usepackage{algorithmic}
\usepackage[title]{appendix}
\usepackage{hyperref}                 
\usepackage{authblk}       		   

\newtheorem{theorem}{Theorem}[section]
\newtheorem{lemma}[theorem]{Lemma}
\newtheorem{corollary}[theorem]{Corollary}

\newtheorem{definition}{Definition}

\newtheorem{remark}{Remark}

\newcommand{\rms}{{\cal{M}}_k(\mathbb{R})}
\newcommand{\sumd}[1]{\sideset{}{'}\sum_{n=0}^{#1}}
\DeclareMathOperator{\tr}{tr}
\DeclareMathOperator{\diag}{diag}
\DeclareMathOperator{\erf}{erf}

\begin{document}

\title{Evaluating Non-Analytic Functions of Matrices}

\author[1]{Nir Sharon\thanks{nsharon@math.princeton.edu} }
\author[2]{Yoel Shkolnisky\thanks{yoelsh@post.tau.ac.il}}
\affil[1]{\footnotesize{The Program in Applied and Computational Mathematics, Princeton University, Princeton NJ, USA}}
\affil[2]{\footnotesize{School of Mathematical sciences, Tel-Aviv University, Tel-Aviv, Israel}}

\date{}

\maketitle
\begin{abstract}
The paper revisits the classical problem of evaluating $f(A)$ for a real function $f$ and a matrix $A$ with real spectrum. The evaluation is based on expanding $f$ in Chebyshev polynomials, and the focus of the paper is to study the convergence rates of these expansions. In particular, we derive bounds on the convergence rates which reveal the relation between the smoothness of $f$ and the diagonalizability of the matrix A. We present several numerical examples to illustrate our analysis. 
\end{abstract}




\section{Introduction}

We revisit the problem of lifting a real function $f \colon \mathbb{R} \to \mathbb{R}$ to a matrix function $f \colon \rms \to \rms$, where $\rms$ is the set of square real matrices of size $k \times k$ having real spectrum. When $f$ is a polynomial, such lifting is straightforward since addition and powers are well-defined for square matrices. When $f$ is not a polynomial, there are several standard methods to define the above-mentioned lifting. If $f$ is analytic having a Taylor expansion whose convergence radius is larger than the spectral radius of $A$, then the Taylor expansion $f(x) =\sum_{n=0}^\infty \alpha_nx^n$ yields $f(A) = \sum_{n=0}^\infty \alpha_nA^n$. If $f$ is not analytic, it is required that at least $f \in C^{m-1}$, where $m$ is the size of the largest Jordan block of $A$. This condition allows defining $f(A)$ on each of the Jordan blocks of $A$. This latter approach has several equivalent definitions, see e.g. \cite[Chapter 1]{higham2008functions}.

Chebyshev polynomials are ubiquitous in applied mathematics and engineering\footnote{``Chebyshev polynomials are everywhere dense in numerical analysis". This quote by Philip Davis and George Forsythe is the opening sentence of \cite{mason2010chebyshev}.}. These polynomials arise as solutions of a Sturm-Liouville ODE and are used in numerous approximation methods, ranging from classical PDE methods~\cite{mason2010chebyshev} to modern methods for image denoising~\cite{may2016algorithm}. Motivated by the favorable numerical properties of Chebyshev polynomials in representing and approximating scalar functions, we rigorously study the use of Chebyshev expansions for matrix functions. The idea of evaluating a matrix function by its Chebyshev expansion is not new and has been used in applications before. For example, \cite{tal1984accurate} uses Chebyshev expapansion in spectral methods for solving PDEs. In this context of solving PDEs, the Chebyshev polynomials are also examined as a special case of ultraspherical polynomials~\cite{doha1991coefficients}, and Faber polynomials~\cite{tal1989polynomial}. Expansion in the Faber polynomials (for complex spectra) appears in~\cite{novati2001computation} for matrices. Chebyshev polynomials are also widely used for pseudospectral methods, see~\cite{trefethen2000spectral} and reference therein. In all of the above papers, one assumes that $f$ is a smooth function. One of the most frequently used analytic function is the exponential function, naturally arising in  solving differential equations. In~\cite{bergamaschi2000efficient} Chebyshev polynomials are proved to be an effective alternative to Krylov techniques for calculating $\exp(A)v$, for a given vector $v$. Another application of Chebyshev polynomials is in representing the best matrix $2$-norm approximation for analytic functions, over the space of polynomials of a fixed degree~\cite{liesen2009best}. In \cite{di2016efficient,saad2011numerical} matrix Chebyshev polynomials are used for slicing the spectrum of a matrix in order to extract interior eigenvalues. Another kind of spectrum filtering is presented in~\cite{may2016algorithm} for the construction of image denoising operators. We can also find matrix Chebyshev polynomials in calculating matrix functions of symmetric matrices~\cite{druskin1989two}, computing square roots of the covariance matrix of Gaussian random fields~\cite{dietrich1995efficient} and facilitating the estimation of autoregressive models~\cite{pace2004chebyshev}.

\subsection*{Our contribution}

In this paper, we generalize the study of matrix Chebyshev expansions to cases where the matrix is not necessarily diagonalizable, and the function is not necessarily analytic. In such cases, there exists a trade-off between two factors; ``how much" the matrix is far from being diagonalizable, as expressed by the size of the largest Jordan block in the Jordan form of the matrix, and ``how smooth is the function''. Specifically, as the size of the largest Jordan block increases (``less diagonalizable") the smoothness of the function required to guarantee the convergence of the matrix Chebyshev expansion, increases as well. 

In the current paper, we mainly focus on examining convergence issues and proving convergence rates. The convergence rate of the matrix Chebyshev expansion is crucial for determining how many coefficients one has to use to approximate $f(A)$ to a prescribed accuracy using a truncated Chebyshev expansion. Thus, the convergence rate directly affects the efficiency of approximation algorithms that use matrix Chebyshev expansions. 

Chebyshev polynomials are naturally define on $[-1,1]$, so without loss of generality, we assume that the spectrum of $A$ is linearly transformed to this segment using an estimated bound on the spectral radius of $A$ (for example, by estimating the eigenvalue of $A$ with the largest magnitude). We henceforth assume that all eigenvalues of $A$ lie in [-1,1]. Given this assumption, we divide our analysis into two cases: a case where there is no a priori information about the distribution of the eigenvalues over $[-1,1]$, and a case where all non-semisimple eigenvalues are concentrated inside $(-1,1)$. As expected, the latter case implies fewer restrictions on the smoothness of the function. In particular, if we denote by $m$ the size of the largest Jordan block of $A$. Then, the convergence of the matrix Chebyshev expansion is guaranteed in the first case for any $f \in C^{2m-2}$ where $f^{(2m-1)}$ is of bounded variation. In the second case, it is guaranteed for any $f \in C^{m-1}$ where $f^{(m)}$ is of bounded variation. As for convergence rates, we summarize our results for the different cases in Table~\ref{tab:convergenceResult}. Each row in the table presents the requirements, from both the matrix and the function, so that an expansion of length $N$ has an error bound of $cN^{-\ell}$ ($\ell>1$). The constant $c$ is independent of $N$ and $\ell$ but does depend on the factors specified in the 'Remarks' column. We note two properties from the table. First, for matrices with a concentrated spectrum, the smoothness assumptions on $f$ are weaker than in the general case. Second, one can guarantee a constant which is independent of the size of the matrix by imposing additional regularity on the function $f$. Further conclusions and their proofs appear in the text. 

 \begin{table}\centering
 \scalebox{0.85}{
  \begin{tabular}{c|c|c}
    Spectrum of $A$  & The function & Remarks \hfill \\ \hline\hline
     & &  \\
    $A$ is a $\delta$-condensed matrix   & $f \in C^{m+\ell-2}$ and $f^{(m+\ell-1)}$ is of bounded variation &  \parbox{6cm}{All non-semisimple \\ eigenvalues of $A$ are in $[-1+\delta,1-\delta]$ \\ The constant depends on $k$, $f$, $m$ } \\
    & &  \\  \hline
    & &  \\
   In $[-1,1]$   & $f \in C^{2m+\ell-3}$ and $f^{(2m+\ell-2)}$ is of bounded variation &  \parbox{6cm}{The constant depends on $k$, $f$, $m$} \\
    & &  \\  \hline
    & &  \\
    $A$ is a $\delta$-condensed matrix   & $f \in C^{m+2\ell-1}$ and $f^{(m+2\ell)}$ is of bounded variation &  \parbox{6cm}{All non-semisimple \\ eigenvalues of $A$ are in $[-1+\delta,1-\delta]$ \\ The constant depends on $f$, $m$ } \\
    & &  \\ \hline
    & &  \\
    In $[-1,1]$ & $f \in C^{2m+2\ell}$ and $f^{(2m+2\ell+1)}$ is of bounded variation &  \parbox{6cm}{The constant depends on $f$ and $m$}  \\ 
  \end{tabular}  }
  \caption{Table of conditions to establish a convergence rate of order $N^{-\ell}$ of a truncated matrix Chebyshev expansion of length $N$. The matrix is of size ${k\times k}$ and with $m$ as the size of its largest Jordan block.}
  \label{tab:convergenceResult}
\end{table}

\parbox{4cm}{}
\subsection*{The structure of the paper}

In Section~\ref{sec:background} we present our notation and the required mathematical background. In Section~\ref{sec:Mat_Cheb_exp} we introduce the matrix Chebyshev expansion and our main theoretical results. These results include the conditions for convergence of a matrix Chebyshev expansion to the matrix function and bounds on the convergence rates. We illustrate some of our theoretical results using numerical examples in Section~\ref{sec:numerics}.


\section{Background and notation} \label{sec:background}

Chebyshev polynomials of the first kind of degree $n$ are defined as
\begin{equation} \label{eqn:Cheby_polynomial_trig_def}
 T_n(x) = \cos \left( n \arccos(x) \right) , \quad x \in [-1,1] , \quad n = 0,1,2,\ldots
\end{equation}
These polynomials are solutions of the Sturm-Liouville ordinary differential equation
\begin{equation} \label{eqn:ODE}
(1-x^2)y^{\prime \prime } - xy^\prime +n^2y=0
\end{equation}
and satisfy the three term recursion
\begin{equation} \label{eqn:three_term_recursion}
  T_n(x) = 2xT_{n-1}(x) - T_{n-2}(x) , \quad n=2,3,\ldots
\end{equation}
with $T_0(x)=1$ and $T_1(x) = x$. Therefore, Chebyshev polynomials form an orthogonal basis for $L_2([-1,1])$ with respect to the inner product
\begin{equation}  \label{eqn:inner_product}
\left\langle  f,g \right\rangle_T =  \frac{2}{\pi} \int_{-1}^1 \frac{f(t)g(t)}{\sqrt{1-t^2}} dt .
\end{equation}
The Chebyshev expansion of a function $f$ with a finite norm with respect to \eqref{eqn:inner_product} is 
\begin{equation} \label{eqn:chebyshev_series_scalars}
 f(x) \sim \sumd{\infty} \alpha_n[f] T_n(x) , \quad \alpha_n[f] = \left\langle f, T_n \right\rangle_T ,
\end{equation}
where the dashed sum \resizebox{.33in}{!}{$\left( \sideset{}{'}\sum \right)$} denotes that the first term is halved. The truncated Chebyshev expansion is defined as 
\begin{equation} \label{eqn:truncated_chebyshev}
  \mathcal{S}_N(f)(x)  = \sumd{N} \alpha_n[f]T_n(x)  .
\end{equation}
$\mathcal{S}_N(f)(x)$ is a polynomial approximation of $f$ which is the best least squares approximation with respect to the induced norm $\norm{f}_T = \sqrt{\left\langle  f,f \right\rangle_T}$. Remarkably, this least squares approximation is close to the best minimax polynomial approximation, measured by the maximum norm $\norm{f}_\infty = \max_{x\in [-1,1]} \abs{f(x)}$. The following theorem was established by Bernstein \cite{bernstein1918quelques} back in 1918 (and was known even before, see references therein).
\begin{theorem}[Bernstein]
For any $f \in C([-1,1])$ and $N \ge 0$
\[ \norm{f(x) - S_N(f)(x)  }_{\infty} \le  \Lambda_N \norm{f(x) - p^\ast_N(x)  }_{\infty} , \]
where $p^\ast_N$ is the unique best minimax polynomial approximation of degree $N$, and $\Lambda_N$ behaves asymptotically as $\log(N)$ for large $N$.
\end{theorem}
The constant $\Lambda_N$ is the Lebesgue constant, and a formula for its exact value has been derived in~\cite{powell1967maximum} (and in particular, it is less than $6$ for $N<1000$). Note that on the boundaries of $[-1,1]$, we only consider one-sided continuity (and derivatives when required). 

To take advantage of recent results for Chebyshev expansions, we introduce the total variation (TV) norm 
\[\norm{g}_{TV} = \norm{g^\prime}_{1} =   \int_{-1}^1 \abs{ g^{\prime}(t)} dt . \]
Note that this norm can be written in a distributional form instead of the derivative form and thus we say that a function has a bounded variation if its TV norm is finite (whether continuous or not, e.g., the step function). Using the TV norm, it was recently shown (see Theorem~\ref{thm:Trefethen} below) that for a smooth $f$, the uniform error $\norm{f(x) - S_N(f)(x)  }_{\infty}$ decays rapidly \cite[Chapter 7]{trefethen2013approximation}.
\begin{theorem}[Trefethen] \label{thm:Trefethen}
Let $f \in C^{m-1}([-1,1])$, such that $f^{(m-1)}$ is absolutely continuous and that $\norm{f^{(m)}}_{TV} <\infty$. Then, for each $N>m$, the Chebyshev coefficients satisfy
\[ \alpha_N[f] \le \frac{C\cdot m}{N(N-1)\cdots(N-m)} \le  (C\cdot m) \frac{1}{(N-m)^{m+1}} ,\]
and moreover,
\[ \norm{f(x) - \mathcal{S}_N(f)(x)  }_{\infty} \le  C \frac{1}{(N-m)^m} , \]
where $C = C(f,m) = \frac{2}{\pi m}\norm{f^{(m)}}_{TV}$.
\end{theorem}

The trigonometric form \eqref{eqn:Cheby_polynomial_trig_def} implies that $T_n(\cos(\theta)) = \cos(n\theta)$, $n=0,1,2,\ldots$, and reveals that the Chebyshev expansion of a given $f(x)$, $x \in [-1,1]$, coincides with the Fourier cosine series $g(\theta) = \sumd{\infty} \alpha_n[f] \cos(n \theta)$ where $g(\theta) = f(\cos(\theta))$ with $\theta \in [0,\pi]$. The argument $\cos(\theta)$ in $f$ makes $g$ even and periodic in $\theta$, which in fact implies that the derivatives of $g$ are also equal on the boundaries, that is $g^{(j)}(\pi) = g^{(j)}(-\pi)$, $j \in \mathbb{N}$. Thus, no periodicity is required from $f$ nor from its derivatives to ensure that $g$ is differentiable and its Fourier series can be differentiated term by term. We conclude the above discussion with the next lemma.
\begin{lemma} \label{lemma:Chebyshev_expansion_prop_from_fourier}
Let $f \in C^{2m-2}([-1,1])$ be such that $f^{(2m-2)} $ is absolutely continuous with $f^{(2m-1)} $ of bounded variation, $m \in \mathbb{N}$. Then,
\begin{enumerate}
\item
The series
\[ \sumd{\infty} \alpha_n[f] T^{(j)}_n(x) , \quad j=0,1,\ldots,m-1   ,\quad x \in [-1,1] ,  \]
is absolutely convergent.
\item
The series~\eqref{eqn:chebyshev_series_scalars} can be differentiated term-by-term, yielding
\[   f^{(j)}(x) = \sumd{\infty} \alpha_n[f] T^{(j)}_n(x) , \quad j=0,1,\ldots,m   ,\quad x \in [-1,1] .  \]
\end{enumerate}
\end{lemma}
The proof of Lemma~\ref{lemma:Chebyshev_expansion_prop_from_fourier} is given for completeness in Appendix \ref{sec:appendix_background_lemma}.


\section{Matrix Chebyshev expansion} \label{sec:Mat_Cheb_exp}

This section presents the theoretical results of the paper; we define a matrix function via its Chebyshev expansion, discuss its convergence, and derive its convergence rate. To this end, unless otherwise stated, we assume a given matrix norm $\norm{\cdot}$ that satisfies the sub-multiplicative property $\norm{XY} \le \norm{X}\norm{Y}$. Equipped with $\norm{\cdot}$, we restrict the convergence analysis to absolute convergence (in norm). Namely, given a series of matrices $\{ X_n \}_{n \in \mathbb{N}} \subset \rms$, we consider the convergence of $\norm{X_n}$ as $n$ tends to infinity. Note that since all matrix norms are equivalent, absolute convergence in one norm implies absolute convergence in any other norm. For more details about the convergence of matrix series and matrix norms, we refer the reader to \cite[Chapter 5]{horn2012matrix}.  

\subsection{Definition and convergence} \label{subsec:Cheby_matrix_series}

In the sequel, we denote by $\lambda_1,\ldots,\lambda_k \in \mathbb{R} $ the eigenvalues of a given matrix $A \in \rms$ and by $\rho(A) = \max_{i=1,\ldots,k} \abs{\lambda_i}$ the spectral radius of $A$. Since we discuss Chebyshev polynomials, we assume that $\rho(A) \le 1$ and that $f$ is a scalar function defined on $[-1,1]$. Also, we denote by $\| X \|_F = \sqrt{ \operatorname{tr}(XX^T) } $ the Frobenius norm of $X$.

We begin by defining the matrix Chebyshev expansion.
\begin{definition}[Matrix Chebyshev expansion]  \label{def:matrix_func_via_Chebyshev}
Given a function $f$ on $[-1,1]$, we define the partial sum 
\[ \mathcal{S}_N (f)(A) = \sumd{N}  \alpha_n[f]T_n(A)  , \]
where $\alpha_n[f]$ is given in~\eqref{eqn:chebyshev_series_scalars}. If the sequence $\mathcal{S}_N(f)(A)$ is absolutely convergent with respect to any matrix norm $\norm{\cdot}$, we call its limit $\mathcal{S}_\infty(f)(A)$ the matrix Chebyshev expansion.
\end{definition}
Using Definition~\ref{def:matrix_func_via_Chebyshev}, we address two fundamental questions: how can one guarantee the (absolute) convergence of $\mathcal{S}_N(A)$? And, does $\mathcal{S}_\infty(f)(A) $ coincide with $f(A)$, whenever the latter is defined, using the standard definitions of matrix functions (see \cite[Chapter 1]{higham2008functions})?

A basic property of polynomials, and so also of the Chebyshev polynomials $T_n$, is that they preserve matrix similarity. Namely, $A = P^{-1}BP$ implies $T_n(A) = P^{-1}T_n(B)P$. Therefore, for a diagonalizable matrix $A$, evaluating $T_n(A)$ reduces to applying $T_n(x)$ on the eigenvalues of $A$. Thus, for such matrices, the convergence of the Chebyshev expansion of a scalar function is inherited by its matrix version. This is summarized in the following corollary. 
\begin{corollary} \label{cor:convergence_normal_func}
Let $f$ be a function on $[-1,1]$ having an absolutely convergent Chebyshev expansion. Then, for any diagonalizable matrix $A$, having $\rho(A)\leq 1$, the matrix Chebyshev expansion of $f(A)$ is convergent.
\end{corollary}
\begin{proof}
The diagonalizability of $A$ means that $A= Q^{-1}DQ$, where $D = \diag(\lambda_1,\ldots,\lambda_k)$. Therefore, $\norm{T_n(A)} \le \norm{Q^{-1}}\norm{T_n(D)}\norm{Q}$. Due to~\eqref{eqn:Cheby_polynomial_trig_def}, $\abs{T_n(\lambda_i)} \le 1$, $1\le i \le k$ which implies that $\norm{T_n(D)}_F \le \sqrt{k}$. Thus, we get
\[ \sumd{\infty} \norm{\alpha_n[f]  T_n(A)}_{F} \le \norm[2]{Q^{-1}}_F\norm[2]{Q}_F \sqrt{k}\sumd{\infty} \abs{\alpha_n[f]} < \infty . \]
\end{proof}

For a general matrix $A \in \rms$, we can use a similar argument but with the Jordan form $A = Z^{-1}JZ$, where $J=\diag{\left ( J_1,\ldots,J_p \right ) }$ is a diagonal block matrix with the blocks $J_1,\ldots,J_p$ on its diagonal. Since applying a polynomial on a block diagonal matrix reduces to applying it to each block separately we get that,
\begin{equation} \label{eqn:chebyshev_polynomial_on_Jordan_form}
  T_n(A) = Z^{-1} \diag{ \left( T_n(J_1),\ldots,T_n(J_p) \right) } Z , \quad n \in \mathbb{N} .
\end{equation} 
Thus, to further explore $T_n(A)$, we focus on $T_n(J_i)$, where $J_i$ is a Jordan block of size $k_i \times k_i$ corresponding to the eigenvalue $\lambda_i$,
\[
 J_i =  \left( \begin{array}{cccc}
\lambda_i & 1 &  & \\
  & \lambda_i & \ddots &  \\
  &   & \ddots & 1 \\
  &   &  & \lambda_i   \\
 \end{array} \right)_{k_i \times k_i} .
\] 
\begin{lemma} \label{lemma:polynomial_on_jordan_block}
Let $p(x)$ be a polynomial. Then,
\begin{equation} \label{eqn:JordanPolynomial}
p(J_i) =  \left( \begin{array}{cccc}
p(\lambda_i) & p^\prime(\lambda_i) & \cdots & \frac{p^{(k_i-1)}(\lambda_i) }{(k_i-1) !} \\
  &  p(\lambda_i)  & \ddots & \vdots  \\
  &   & \ddots &  p^\prime(\lambda_i) \\
  &   &  & p(\lambda_i)   \\
 \end{array} \right) . 
\end{equation}
\end{lemma}
The detailed proof of Lemma~\ref{lemma:polynomial_on_jordan_block} is omitted as it follows directly from the upper triangular entries of $(J_i)^n$, to wit $(J_i^n)_{q,j}=\binom{n}{j-q}\lambda_i^{n-j+q}$. Obviously, this result agrees with standard definitions of matrix functions, see e.g., see \cite[Definition 1.2]{higham2008functions}.

The explicit form of $T_n(J_i)$, derived from Lemma \ref{lemma:polynomial_on_jordan_block}, gives rise to the next convergence result which generalizes Corollary~\ref{cor:convergence_normal_func} to non-diagonalizable matrices. For proving this generalization we need a bound on the derivatives of the Chebyshev polynomials. In \cite[Chapter 1.5]{rivlin1974chebyshev} it is proven that for any $j \in \mathbb{N}$
\begin{equation} \label{eqn:jordan_block_form}
 \abs{T_n^{(j)}(x)} \le \abs{T_n^{(j)}(1)} = \frac{n^2(n^2-1)\cdots(n^2-(j-1)^2)}{(2j-1)\cdots 5 \cdot 3 \cdot 1} \le \frac{n^{2j}}{(2j-1)!}.
\end{equation}
Therefore we have,

\begin{theorem} \label{thm:convergence_cheby}
Let $A\in\rms$ and let $m \ge 1$ be the size of the largest Jordan block in the Jordan form of $A$. Assume that $f \in C^{2m-2}([-1,1])$ such that $f^{(2m-2)}$ is absolutely continuous and $f^{(2m-1)}$ is of bounded variation. Then, the sequence $\mathcal{S}_N(f)(A)$ is absolutely convergent.
\end{theorem}
\begin{proof}
It is convenient to prove the theorem using the $\ell_1$-induced matrix norm, $\norm{X}_1 = \max_i \sum_j \abs{X_{i,j}}$ whereby we get
\begin{equation} \label{eqn:bound_mat_cheby_polynomial}
 \norm{T_n(A)}_1 \le  \norm[2]{Z^{-1}}_1 \left(\max_{i=1,\ldots,p} \norm[2]{T_n(J_i)}_1 \right) \norm[2]	{Z}_1 , \quad n \in \mathbb{N} .
\end{equation}
By \eqref{eqn:JordanPolynomial}, for any Jordan block $J_i$ of size $k_i \times k_i$
\begin{equation*}  
 \norm{T_n(J_i)}_1 = \sum_{j=1}^{k_i} \abs{  \frac{1}{(j-1)!} \frac{d^{j-1}}{dx^{j-1}} T_n(\lambda_i) } , \quad n \in \mathbb{N}.
\end{equation*}  
 Therefore,
\begin{equation} \label{eqn:bound_Jordan_mat_cheby_polynomial_2}
   \norm{T_n(J_i)}_1 \le 1+ \sum_{j=1}^{k_i-1}   \frac{1}{j!} \abs{ T^{(j)}_n(\lambda_i)} \le n^{2k_i-2} \left(1+ \sum_{j=0}^\infty 2^{-j} \right)\le 3n^{2k_i-2}  , \quad n \in \mathbb{N}.
\end{equation}
Recall that $m = \max_{i=1,\ldots,p} k_i$ and combine \eqref{eqn:bound_mat_cheby_polynomial} and \eqref{eqn:bound_Jordan_mat_cheby_polynomial_2} to get
\begin{equation} \label{eqn:TnA_norm1_bound}
 \norm{T_n(A)}_1 \le  \norm[2]{Z^{-1}}_1 \norm[2]{Z}_1 \max_{i=1,\ldots,p} 3n^{2k_i-2} \le C_A n^{2m-2} ,
\end{equation}
where $C_A = 3 \norm{Z^{-1}}_1 \norm{Z}_1 $ is a constant that depends on the Jordan form of $A$ but is independent of $n$. By Theorem \ref{thm:Trefethen}, there exists $C_f$ such that $\abs{\alpha_n[f]} \le C_f n^{-2m} $ for large enough $n$, say $n \ge n^\ast$. Therefore, we have that
\begin{equation} \label{eqn:finite_norm_series}
 \sum_{n=n^\ast}^\infty \norm{\alpha_n[f]  T_n(A)}_1 = \sum_{n=n^\ast}^\infty \abs{\alpha_n[f]}  \norm{T_n(A)}_1 \le C_f C_A \sum_{n=n^\ast}^\infty  n^{-2} < \infty ,
\end{equation}
and so $\mathcal{S}_N(f)(A)$ (see Definition~\ref{def:matrix_func_via_Chebyshev}) is absolutely convergent.
\end{proof}

We next show that $\mathcal{S}_\infty(f)(A)$ coincides with the standard definitions for $f(A)$, see e.g., \cite[Chapter 1]{higham2008functions}.
\begin{theorem} \label{thm:matrixChebyConvergesTofA}
Let $A\in\rms$ and let $f$ satisfy the conditions for an absolutely convergent matrix Chebyshev expansion. Then, $\mathcal{S}_\infty(f)(A)=f(A)$, where $f(A)$ is one of the standard definitions for lifting a scalar function to matrices.
\end{theorem}
\begin{proof}
Since all standard definitions for lifting a scalar function to matrices are equivalent, we only prove that the definition based on matrix Chebyshev expansions is equivalent to the definition based on the Jordan form, see \cite[Definition 1.2]{higham2008functions}. We use the notation of the proof of Theorem \ref{thm:convergence_cheby} and by \eqref{eqn:chebyshev_polynomial_on_Jordan_form} deduce that it is enough to show that the definitions coincide on the Jordan blocks of $A$. Note that in the special case where the matrix is diagonalizable, as in Corollary \ref{cor:convergence_normal_func}, the definitions are equivalent. This is true due to the absolute convergence of the Chebyshev expansion for scalars which implies pointwise convergence, together with the fact that the function $f$ is applied element-wise to the diagonal form of $A$.

Denote by $m$ the size of the largest Jordan block of $A$. An element in the matrix $\mathcal{S}_\infty(f)(J_i)$, corresponding to the Jordan block $J_i$ with eigenvalue $\lambda_i$, is given for every $1 \le p \le j \le m$ by
\begin{equation} \label{eqn:Jordan_block_action1}
 \left( \mathcal{S}_\infty(f)(J_i) \right)_{p,j} = \left( \sumd{\infty}\alpha_n[f]T_n(J_i) \right)_{p,j} =  \sumd{\infty}\alpha_n[f]  \left( T_n(J_i) \right)_{p,j}.
\end{equation}
Note that the absolute convergence of the matrix Chebyshev expansion implies the absolute and uniform convergence of the scalar Chebyshev expansion of $f$. Thus, one can differentiate the Chebyshev expansion of $f$ term-by-term to get $f^{(j)}$, $j=0,\ldots,m-1$ (see Lemma \ref{lemma:Chebyshev_expansion_prop_from_fourier}). Since Lemma \ref{lemma:polynomial_on_jordan_block} states that
\[ \left( T_n(J) \right)_{p,j} = \frac{1}{(j-p)!}  \frac{d^{j-p}}{dx^{j-p}}T_n(\lambda_i) , \]
we get that for $1 \le p \le j \le k_i \le m$
\[ \left( \mathcal{S}_\infty(f)(J_i) \right)_{p,j} = \frac{1}{(j-p)!} \sumd{\infty}  \alpha_n[f]   \frac{d^{j-p}}{dx^{j-p}}T_n(\lambda_i) =  \frac{1}{(j-p)!} f^{(j-p)}(\lambda_i) , \]
which coincides, as required, with the definition according to the Jordan form (see \cite[Definition 1.2]{higham2008functions} and the polynomial version in~\eqref{eqn:JordanPolynomial}).
\end{proof}

\subsection{Relaxing the smoothness requirements} \label{subsec:tighten}

The proof of Theorem \ref{thm:convergence_cheby} relies upon the bound \eqref{eqn:jordan_block_form} that can be improved for cases where the eigenvalues of $A$ lie strictly inside the interval $[-1,1]$. We start by examining the first derivative of $T_n(x)$, which by \eqref{eqn:Cheby_polynomial_trig_def} is $T^\prime_n(x) = n\frac{\sin(n\theta)}{\sin(\theta)}$ for $x = \cos(\theta)$. Therefore, by \cite{lepson1974upper}
\begin{equation} \label{eqn:first_derivative_bound}
 \abs{T^\prime_n(x)} \le \frac{n}{\sqrt{1-x^2}} , \quad \abs{x}<1 , \quad n=0,1,2,\ldots .
\end{equation}
The factor $ \frac{1}{\sqrt{1-x^2}}$ increases near the end points of $[-1,1]$ and ultimately forces us to use the global bound of $n^2$, see e.g., \cite[Chapter 1]{rivlin1974chebyshev}. However, for a segment $I_\delta = [-1+\delta,1-\delta]$, with a fixed $0<\delta<1$, we have
\begin{equation} \label{eqn:gamma_constant}
\abs{T^\prime_n(x)} \le \nu n, \quad \nu = \frac{1}{\sqrt{2\delta(1 -\delta)}} .
\end{equation}
 
For higher derivatives, one can use \eqref{eqn:ODE} to derive \cite[Chapter 1, p. 32--33]{rivlin1974chebyshev}
\[ (1-x^2) T_n^{(k+1)}(x) -  (2k-1)xT_n^{(k)}(x) + (n^2-(k-1)^2)T_n^{(k-1)}(x) = 0  , \quad 1 \le k \le n . \]
Thus, we get
\begin{equation} \label{eqn:high_order_derivatives}
\abs{T_n^{(k+1)}(x)} \le  \frac{(2k-1)\abs{x}}{(1-x^2)}T_n^{(k)}(x)  + \frac{n^2}{1-x^2} \abs{T_n^{(k-1)}(x)} , \quad  \abs{x}<1 , \quad 1 \le k \le n .
\end{equation}
The latter leads to the following bound on higher order derivatives of $T_{n}(x)$.
\begin{lemma} \label{lemma:leading_derivative_order}
Let $0<\delta<1$ and $1 \le k \le n$. Then, the $k-$th derivative of $T_{n}(x)$ satisfies
\[ \abs{T_n^{(k)}(x)} \le c_kn^k , \quad x \in I_\delta, \]
where $c_k = c_k(\nu,k)$ is a constant independent of $n$, and $\nu$ is given by \eqref{eqn:gamma_constant}.
\end{lemma}
\begin{proof}
We prove the lemma by induction on the order $k$ of the derivative. For the basis of our induction we have $\abs{T_n^{(0)}(x)} \le 1 = (\nu n)^0$. For $k=1$ the lemma holds due to \eqref{eqn:first_derivative_bound} and the fact that $x \in I_\delta $. Now, assume the claim is true for any $j$ that satisfies $0 \le j \le k <n $, for a fixed $k$. Then, for $k+1$ we get by~\eqref{eqn:high_order_derivatives} a leading order term $n^{k+1}$ from the leading order of $\abs{T_n^{(k-1)}(x)}$. Namely, $\abs{T_n^{(k+1)}(x)}  \le  n^2 \nu^2 (\nu n )^{k-1} + cn^k$. It is understood from \eqref{eqn:high_order_derivatives} that the constant $c$ of the lower order term (LOT) depends on $k$ and $\nu$ but not on $n$.
\end{proof}

\begin{remark} \label{rmk:LOT}
We did not elaborate on the LOTs of the bound on $T_n^{(k+1)}(x)$, which appear in the proof as $cn^{k}$. For example, with further induction, one can derive that the constant in front of the $n^{k}$ term is bounded by $\frac{(k+1)k}{2} \nu^{k+2}$. The proof is given for completeness in Appendix~\ref{apx:second_LOT}. Nevertheless, the explicit forms of the constants of the LOTs are less important for us, as we aim to pose convergence results, which are typically stated for large $n$. Thus, the explicit forms of the LOTs are omitted.
\end{remark}

Lemma~\ref{lemma:leading_derivative_order} shows that in many cases the bound of $n^2$ for $T^\prime_n(x)$ is very pessimistic and so are the conditions of Theorem~\ref{thm:convergence_cheby} above. We therefore define a class of matrices for which we can relax the smoothness requirements imposed on the function. 
\begin{definition}[$\delta$-condensed matrix]
 Let $0<\delta<1$. We call a matrix $A\in \rms$ with $\rho(A)\le 1$ a {\it $\delta$-condensed matrix} if any eigenvalue in $[-1,-\delta]\cap[\delta,1]$ is semisimple (its algebraic multiplicity is equal to its geometric multiplicity).
\end{definition}

Having the above definition, we get the next convergence result.
\begin{corollary} \label{cor:convergence_mid_segment}
Let $A\in\rms$ be a $\delta$-condensed matrix. Denote by $m \ge 1$ the size of the largest Jordan block in the Jordan form of $A$, and assume that $f \in C^{m-1}([-1,1])$ such that $f^{(m-1)}$ is absolutely continuous and $f^{(m)}$ of bounded variation. Then, the sequence $\mathcal{S}_N(f)(A)$ is absolutely convergent.
\end{corollary}
\begin{proof}
We follow the proof of Theorem \ref{thm:convergence_cheby} and modify it as described next. First, using Lemma~\ref{lemma:leading_derivative_order} we replace \eqref{eqn:bound_Jordan_mat_cheby_polynomial_2} with
\[
  \norm{T_n(J_i)}_1 \le  1+ \sum_{j=1}^{k_i-1}   \frac{1}{(j)!} \abs{ T^{(j)}_n(\lambda_i)} \le  1+ \sum_{j=1}^{k_i-1}   \frac{1}{(j)!}c_jn^j \le \widetilde{c_{k_i}}n^{k_i-1} ,
\]
for a constant $\widetilde{c_{k_i}}$ that depends on $k_i$ and $\nu$ but is independent of $n$. Since by definition $m = \max_{i=1,\ldots, p} k_i$, we can bound 
\[ \norm{T_n(A)}_1 \le Kn^{m-1} , \quad K = \max_{ i=1,\ldots, p} \widetilde{c_{k_i}} \norm[2]{Z^{-1}}_1 \norm[2]{Z}_1 . \]
Here $K$ is a constant that depends only on $\nu$ and on the Jordan form of $A$. By Theorem \ref{thm:Trefethen}, the assumptions on $f$ ensure that \eqref{eqn:finite_norm_series} still holds.
\end{proof}

\begin{remark} \label{rmk:updateLemma}
Using the same arguments as above, we can also relax the conditions of Lemma~\ref{lemma:Chebyshev_expansion_prop_from_fourier} to $f \in C^{m-1}([-1,1])$ such that $f^{(m-1)}$ is absolutely continuous and $f^{(m)}$ is of bounded variation, for cases where $x\in [1+\delta,1-\delta]$ with a positive $\delta<1$.
\end{remark}

\subsection{The truncated matrix Chebyshev expansion} \label{subsec:truncatedMatrixCheby}

The convergence rates of the Chebyshev expansion for scalar functions are well-studied (see Section~\ref{sec:background}), and we wish to extend some of the results to the matrix case. In particular, we generalize Theorem~\ref{thm:Trefethen} to the case of matrix functions.

In the case of a diagonalizable matrix $A = Q^{-1}DQ$, we bound the convergence rate of $\mathcal{S}_N(f)(A)$ based on the convergence rate of the scalar function $f$ as
\[ \norm{\mathcal{S}_N(f)(A)-f(A)}  \le \norm[1]{Q^{-1}} \norm[1]{ \mathcal{S}_N(f)(D) - f(D)}\norm[1]{Q} , \]
where the Chebyshev expansion is applied to the diagonal matrix $D$ element-wise. Next, we proceed to more general cases where the matrices are allowed to be non-diagonalizable. In the following result we show how the smoothness of $f$ is related to the convergence rate of its expansion. Specifically, if $f$ is differentiable $\ell$ times more than required by Theorem~\ref{thm:convergence_cheby}, then the difference between $\mathcal{S}_N(f)(A)$ and $f(A)$ converges to zero as a power of $\ell-1$.
\begin{theorem} \label{thm:firstAppRes}
Let $A\in\rms$ and let $m \ge 1$ be the size of the largest Jordan block in the Jordan form of $A$. Assume $f \in C^{2m-2+\ell}([-1,1])$ such that $f^{(2m-2+\ell)}$ is absolutely continuous and $f^{(2m-1+\ell)}$ is of bounded variation. Then, for a given matrix norm $\norm{\cdot}$ we have
\[ \norm{\mathcal{S}_N(f)(A) - f(A)} \le C \frac{1}{(N-(2m-1+\ell))^{\ell+1}} , \quad N>2m-1+\ell , \]
where the constant $C = C(f,\ell,m,\norm{\cdot})$ is independent of $N$.
\end{theorem} 
\begin{proof}
Since by Theorem~\ref{thm:convergence_cheby} and Theorem~\ref{thm:matrixChebyConvergesTofA} $\mathcal{S}_N(f)(A)$ converges in norm to $f(A)$ we have 
\[ \norm{\mathcal{S}_N(f)(A) - f(A)}  =  \norm{ \sum_{n=N+1}^\infty \alpha_n[f] T_n(A) }  ,  \]
which can be bounded by $\sum_{n=N+1}^\infty \abs{\alpha_n[f]}\norm{T_n(A)}$. By the equivalence of norms we have that $\norm{\cdot} \le C_{\text{norm}} \norm{\cdot}_1$ for some constant $C_{\text{norm}}$, and so by \eqref{eqn:TnA_norm1_bound} 
\begin{equation} \label{eqn:AnotherTnABound}
 \norm{T_n(A)} \le C_{\text{norm}}C_An^{2m-2} .
\end{equation}
The assumptions on $f$ and Theorem~\ref{thm:Trefethen} imply that
\begin{equation} \label{eqn:BndOnAlphan}
 \abs{\alpha_n[f]} \le \frac{2V}{\pi} \frac{1}{n(n-1)\cdots(n-(2m-1+\ell))} ,\quad n>2m-1+\ell , 
\end{equation}
with $V=\norm{f^{(2m-1+\ell)}}_{TV}$. Note that $\frac{n^{2m-2}}{n(n-1)\cdots(n-(2m-1+\ell))}$ can be simplified as
\[ \frac{n^{2m-3}}{(n-1)\cdots(n-(2m-3))} \frac{1}{(n-(2m-2)\cdots(n-(2m-1+\ell))}  , \]
which is bounded by
 \[  \left( \frac{n}{n-(2m-3)} \right)^{2m-3} \left( \frac{1}{n-(2m-1+\ell)} \right)^{\ell+2} . \]
For $n>2m-1+\ell$ we have
\[ \frac{n}{n-(2m-3)} = 1+ \frac{2m-3}{n-(2m-3)} \le 1+\frac{2m-3}{\ell+2}. \]
Therefore, we denote $C = C_{\text{norm}}C_A(\frac{2V}{\pi}) \left( 1+\frac{2m-3}{\ell+2} \right)^{2m-3}$ and combine \eqref{eqn:AnotherTnABound} and~\eqref{eqn:BndOnAlphan} to get
\begin{eqnarray*}
\sum_{n=N+1}^\infty \abs{\alpha_n[f]}\norm{T_n(A)} &\le C& \sum_{n=N+1}^\infty  \left( \frac{1}{n-(2m-1+\ell)} \right)^{\ell+2}  \\
&\le &  C \int_N^\infty \frac{dx}{(x-(2m-1+\ell))^{(\ell+2)}} \\
& = & \left( \frac{C}{\ell+2} \right) \frac{1}{(N-(2m-1+\ell))^{\ell+1}} .
\end{eqnarray*}
\end{proof}
A different approach to prove Theorem~\ref{thm:firstAppRes} is via element-wise estimation. In particular, we start by reducing the required bound to the case of a Jordan block, similar to~\eqref{eqn:chebyshev_polynomial_on_Jordan_form}, as 
\[ \mathcal{S}_N(f)(A) - f(A) = Z^{-1}\left( \sum_{n=0}^N \alpha_n[f]  \diag{ \left( T_n(J_1),\ldots,T_n(J_p) \right)} -   \diag{ \left( J_1,\ldots,J_p \right)}  \right)Z .\]
Then, the element-wise difference is measured for each block separately, and is of the form
\begin{equation} \label{eqn:error_scalar_derivs}
 f^{(j-1)}(\lambda_i) - \sum_{n=0}^N \alpha_n[f] T_n^{(j-1)}(\lambda_i) = \sum_{n=N+1}^\infty \alpha_n[f] T_n^{(j-1)}(\lambda_i) , \quad 0 \le j <m \quad , 1 \le i \le p. 
\end{equation}
Now, we can use the bound~\eqref{eqn:BndOnAlphan} and the following derivation in the current proof to establish a bound for~\eqref{eqn:error_scalar_derivs}. Finally, one can cast the resulting error in terms of any desired matrix norm. 

A different adjustment to the proof of Theorem~\ref{thm:firstAppRes} is for the case of $\delta$-condensed matrices. Specifically,~\eqref{eqn:AnotherTnABound} becomes $\norm{T_n(A)} \le C_{\text{norm}}C_A c_{m-1} n^{m-1}$. Therefore, the requirements from $f$ can be weakened and the remaining calculations of the proof hold with $m-1$ replacing $2m-2$. We summarize this in the following corollary.
\begin{corollary} \label{cor:firstConvergenceCondenceMat}
Let $A\in\rms$ be a $\delta$-condensed matrix and let $m>1$ be the size of the largest Jordan block in its Jordan form. Assume $f \in C^{m-1+\ell}([-1,1])$ such that $f^{(m-1+\ell)}$ is absolutely continuous and $f^{(m+\ell)}$ is of bounded variation. Then, for a given matrix norm we have
\[ \norm{\mathcal{S}_N(f)(A) - f(A)} \le C \frac{1}{(N-(m+\ell))^{\ell+1}} , \quad N>m+\ell , \]
where the constant $C = C(f,\ell,m,\norm{\cdot})$ is independent of $N$.
\end{corollary}

\begin{remark}
There are some intermediate cases between the two cases presented in Theorem~\ref{thm:firstAppRes} and Corollary~\ref{cor:firstConvergenceCondenceMat}. Specifically, consider a matrix $A$ that has an eigenvalue of $1$ or $-1$, which is not semisimple, but whose associated Jordan block is of size $\widetilde{m}$ with $\widetilde{m}<m$. That means that the power of $n$ in the bound~\eqref{eqn:TnA_norm1_bound} is determined by the maximum between $m-1$ and $2\widetilde{m}-2$. Then, the smoothness requirements from $f$ should be modified accordingly. The proof of such a result can be easily recovered from the current section and is therefore omitted.
\end{remark}

Theorem~\ref{thm:firstAppRes} shows a way to lift convergence rate results from scalars to matrices. However, two weaknesses of this approach are hidden in the constant $C$. First, $C$ has $C_A$ as a factor, namely it depends on the condition number of $Z$ (see \eqref{eqn:TnA_norm1_bound}). This condition number can be significantly large. Second, $C$ depends on the constants relating different matrix norms, which typically depend on the size of $A$.

\subsection{Bounds which are independent of the matrix size}
Inspired by \cite{mathias1993approximation}, we suggest an alternative approach for lifting convergence rate results from scalars to matrices, based on the duality theorem for matrix norms, see e.g., \cite[Chapter 3]{horn1991topics}. The bound we get using this approach is independent of the size of the matrix but requires additional smoothness from~$f$. This means, for example, that for a given $f$ the convergence rate for a matrix consisting of one Jordan block is the same as for a matrix consisting of many copies of this block.
\begin{theorem} \label{thm:convergence_rate}
Let $A\in\rms$ and denote by $m$ the size of the largest Jordan block in the Jordan form of~$A$. Assume $f \in C^{2m+2\ell}([-1,1])$ such that $f^{(2m+2\ell)}$ is absolutely continuous and $f^{(2m+2\ell+1)}$ is of bounded variation. Then, for a given matrix norm we have
\[ \norm{\mathcal{S}_N(f)(A) - f(A)} \le C \frac{1}{(N-\ell)^{\ell}} , \quad N>\ell , \]
where $C$ is a constant independent of $N$ and $k$. This constant can be bounded by 
\[ \frac{2}{\ell} \int_{-1}^1 \norm{ \sumd{\infty} \alpha_n[f] T_n(A) T^{(\ell+1)}_n(t)} dt . \] 
\end{theorem}
The proof of Theorem \ref{thm:convergence_rate} is given in Appendix~\ref{apx:proofTruncError}, where we lift Theorem~\ref{thm:Trefethen} to matrices using an auxiliary Chebyshev operator, acting on matrix-valued functions. Similar approaches were introduced in \cite[Chapter 6]{horn1991topics} and \cite{mathias1993approximation} for other types of matrix operators. 

We conclude this section with a variant of Theorem~\ref{thm:convergence_rate} for $\delta$-condensed matrices. The proof uses arguments similar to those of Corollary~\ref{cor:convergence_mid_segment} and Corollary~\ref{cor:firstConvergenceCondenceMat} and thus is omitted.
\begin{corollary}
 Let $A\in\rms$ be a $\delta$-condensed matrix and denote by $m$ the size of the largest Jordan block in the Jordan form of $A$. Assume $f \in C^{m+\ell-1}([-1,1])$ such that $f^{(m+\ell-1)}$ is absolutely continuous and $f^{(m+\ell)}$ is of bounded variation. Then, for a given matrix norm we have
\[ \norm{\mathcal{S}_N(f)(A) - f(A)} \le C \frac{1}{(N-\ell)^{\ell}} , \quad N>\ell , \]
where $C$ is a constant independent of $N$ and $k$.
\end{corollary}


\section{Numerical examples} \label{sec:numerics}

We turn to demonstrate numerically the theory devoloped in Section~\ref{sec:Mat_Cheb_exp}. This section is divided into four parts. The first part briefly explains our implementation of matrix Chebyshev expansions. The second part demonstrates the evaluation of non-analytic matrix functions by matrix Chebyshev expansions. The third part illustrates numerically the convergence properties of matrix Chebyshev expansions for Jordan block matrices. The fourth part shows an application that uses the numerical advatages of the matrix Chebyshev expansion. In the context of numerical codes, it is worth mentioning the Chebfun software system and resources \cite{driscoll2014chebfun}, where the interested reader can further study other numerical implementations and issues related to Chebyshev expansions.

\subsection{Notes on implementation}

The common practice for approximating a function using Chebyshev expansions consists of computing the expansion coefficients $\alpha_n$ in~\eqref{eqn:chebyshev_series_scalars}, followed by computing the finite sum $\mathcal{S}_N$ of~\eqref{eqn:truncated_chebyshev}. Calculating the expansion coefficients uses the Clenshaw-Curtis quadrature for~\eqref{eqn:inner_product} and is implemented efficiently using the Fast Fourier Transform (FFT) algorithm, see e.g., \cite[Chapter 5.8]{vetterling1992numerical}.

Having obtained the expansion coefficients for $\mathcal{S}_N$ of \eqref{eqn:truncated_chebyshev}, evaluating $\mathcal{S}_N$ employs Clenshaw's algorithm (see e.g., \cite[Chapter 5.5]{vetterling1992numerical}), which exploits the three term recursion \eqref{eqn:three_term_recursion}, as described in \cite[p. 193]{vetterling1992numerical}. Note that this algorithm is valid also for matrices since each polynomial consists only of powers of $A$, which means that any two such matrix polynomials of $A$ commute. Commutativity makes scalar algorithms, such as Clenshaw's algorithm, applicable to matrices. The algorithm for evaluating matrix Chebyshev expansions is given for completeness as Algorithm~\ref{alg:clenshaw} in Appendix \ref{apx:algorithm}.

To substitute a matrix $A$ into a Chebyshev expansion, the eigenvalues of $A$ must be in $[-1,1]$. This constraint can be easily removed by estimating a norm of $A$, for example, the Frobenius norm, and scaling down the matrix. Having a scale factor, say $b$, that guarantees that the scaled-down matrix $\widetilde{A} = A/b$ satisfies $\norm{\widetilde{A}} \le 1$,  we expand $g(\widetilde{A})$ for $g(\cdot) = f(b\cdot)$. Observe that in this case, any theoretical result must be applied to the scaled function.

The entire code for the numerical examples is available online at \url{https://github.com/nirsharon/matrix-Chebyshev-expansion}. The numerical examples were executed in Matlab 2015a on a Springdale Linux desktop with a 3.2~GHz i7-Intel Core\textsuperscript{TM} CPU having $16$ GB of RAM.

\subsection{Lifting non-analytic functions}

An important motivation for using Chebyshev expansions is applying non-analytic functions to matrices. A textbook solution involving diagonalizing $A$ might be computationally infeasible for large matrices, and numerically unstable for general matrices. The same holds if $A$ cannot be diagonalized, and its Jordan decomposition is used instead. Thus, we consider the matrix Chebyshev expansion as an alternative.

In our first set of examples, we use three different functions which we apply on symmetric random matrices of size $10\times 10$. These matrices have, with high probability, only simple eigenvalues (geometric multiplicity of one) and are diagonalizable due to symmetry. We use this simple setting to demonstrate the convergence rates of Subsection~\ref{subsec:truncatedMatrixCheby}. 

The first function we use is
\begin{equation} \label{eqn:f1_numerics}
f_1(x) = \operatorname{sign}(x)x^2  = \left \{ \begin{array}{rl}
               -x^2               & x < 0 , \\
               x^2               & \text{otherwise} .
           \end{array} \right.
\end{equation}
This function is $C^1$, with jump discontinuity in the second derivative at $x=0$. Clearly, no Taylor series is available for evaluating $f_1(A)$. Since the second derivative is of bounded variation, we expect from Theorem~\ref{thm:firstAppRes} that the error of the truncated matrix Chebyshev expansion of length $N$ would decay as $N^{-2}$ (take $m=1$ and $j=1$). To numerically validate this theoretical result, we calculate matrix Chebyshev expansions with up to $2000$ coefficients. $f_1$ is an odd function and thus only half of its expansion coefficients are nonzero. The error in evaluating matrix functions is measured using the spectral norm and we normalize it by the condition number of the matrix,
\[ \norm{\mathcal{S}_N(f)(A)-f(A)}/\left( \norm[1]{A}\norm[1]{A^{-1}}  \right) . \]
The normalization compensate for the natural loss of accuracy due to the condition number of $A$ and to provide a clear view of the convergence rates which we want to test here. We present in Figure~\ref{subfig:non_smooth1} the coefficients' decay, the approximation error of the truncated matrix Chebyshev expansions, and the theoretical bound $N^{-2}$ on a logarithmic scale, as functions of the expansion length~$N$. Indeed, we see that the approximation errors fit the theoretical bound.

We repeat the last experiment with the continuous function
\begin{equation} \label{eqn:f2_numerics}
f_2(x) = \sqrt{\abs{x}} .
\end{equation}
The derivative of this function has a singularity at $x=0$, nevertheless it is of bounded variation. As such, its coefficients' decay is slower than of $f_1$ and so we expect the convergence rate to behave like $N^{-1}$. As in the previous example, we show on a logarithmic scale the coefficients' decay, the convergence rate and its theoretical bound. As seen in Figure~\ref{subfig:non_smooth_second}, the numerical convergence rate fits the expected theoretical rate.

\begin{figure}[ht]
    \centering
\subfloat[$f_1(x)$ of \eqref{eqn:f1_numerics} ]{	  \label{subfig:non_smooth1}  \includegraphics[width=0.4\textwidth]{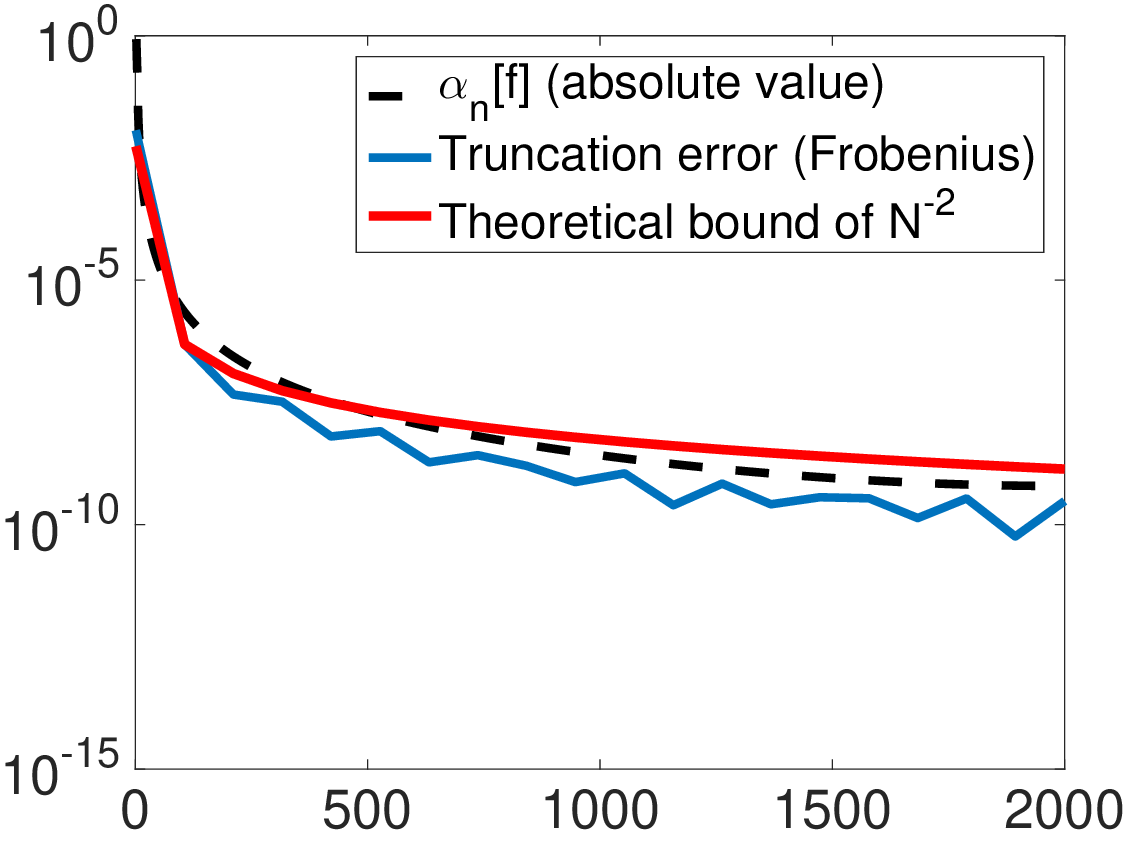}} \qquad
\subfloat[$f_2(x)$ of \eqref{eqn:f2_numerics}]{	\label{subfig:non_smooth_second}    \includegraphics[width=0.4\textwidth]{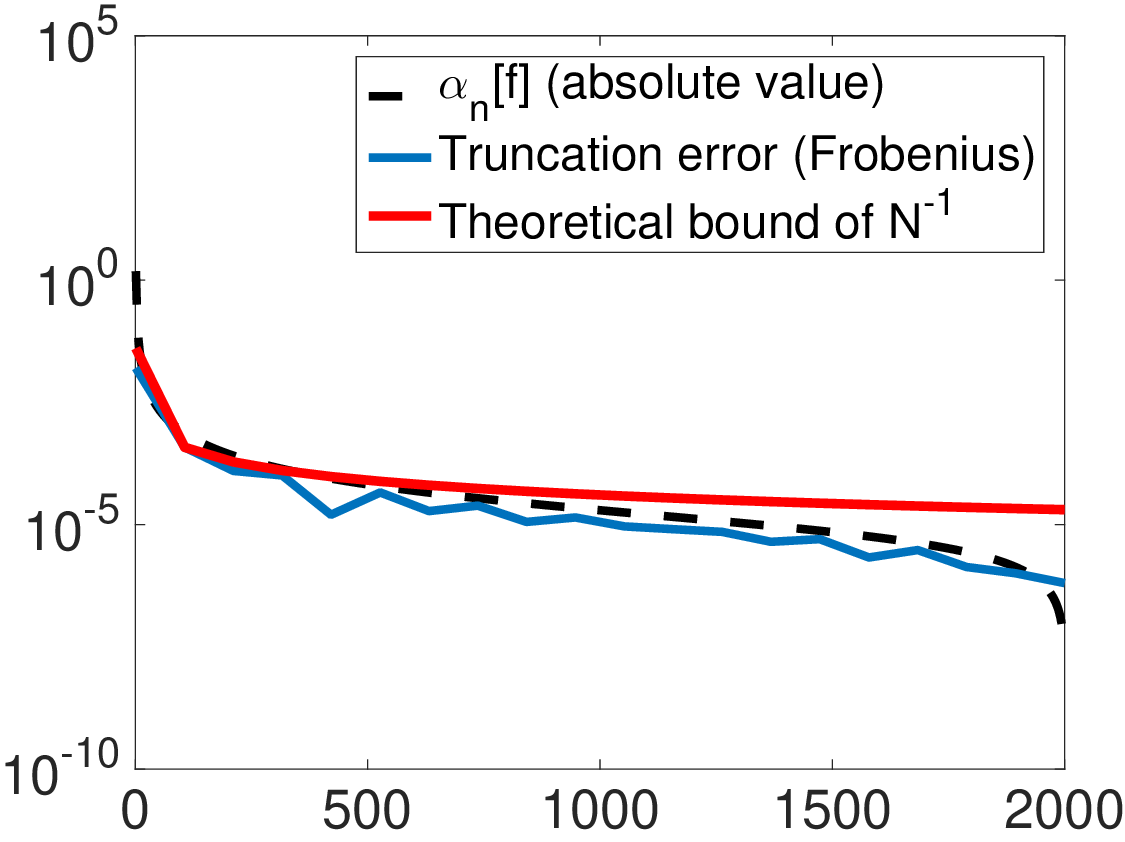}} 
\caption{Lifting scalar functions to $10 \times 10$ random matrices: comparisons of the coefficients' decay, the approximation error of the truncated matrix Chebyshev expansion, and the theoretical bound of Theorem~\ref{thm:firstAppRes}.}
            \label{fig:non_smooth_func}
\end{figure}

As a last example for this part, consider
\begin{equation} \label{eqn:f3_numerics}
f_3(x) = \frac{1}{x^2+0.25} .
\end{equation}
This function is analytic around zero, but with radius of convergence of $0.5$, due to its poles at $\pm 0.5 \dot{\imath}$. Therefore, the Taylor expansion of $f_{3}(x)$ cannot be used for matrices having eigenvalues whose magnitudes are above $0.5$. The matrix that we use is of size $10 \times 10$, with seven eigenvalues whose magnitudes are larger than $0.5$, as seen in the bar plot of Figure~\ref{subfig:f3_spec}. Since $f_3$ has derivatives of any order everywhere on the real line, the coefficients decay rapidly and we have to use only $73$ coefficients (among them only $37$ are nonzero) to get double precision accuracy. Figure~\ref{subfig:f3} shows that indeed both the expansion coefficients and the approximation error decay rapidly.

\begin{figure}[ht]
    \centering
\subfloat[The spectrum of the matrix]{\label{subfig:f3_spec}  \includegraphics[width=0.4\textwidth]{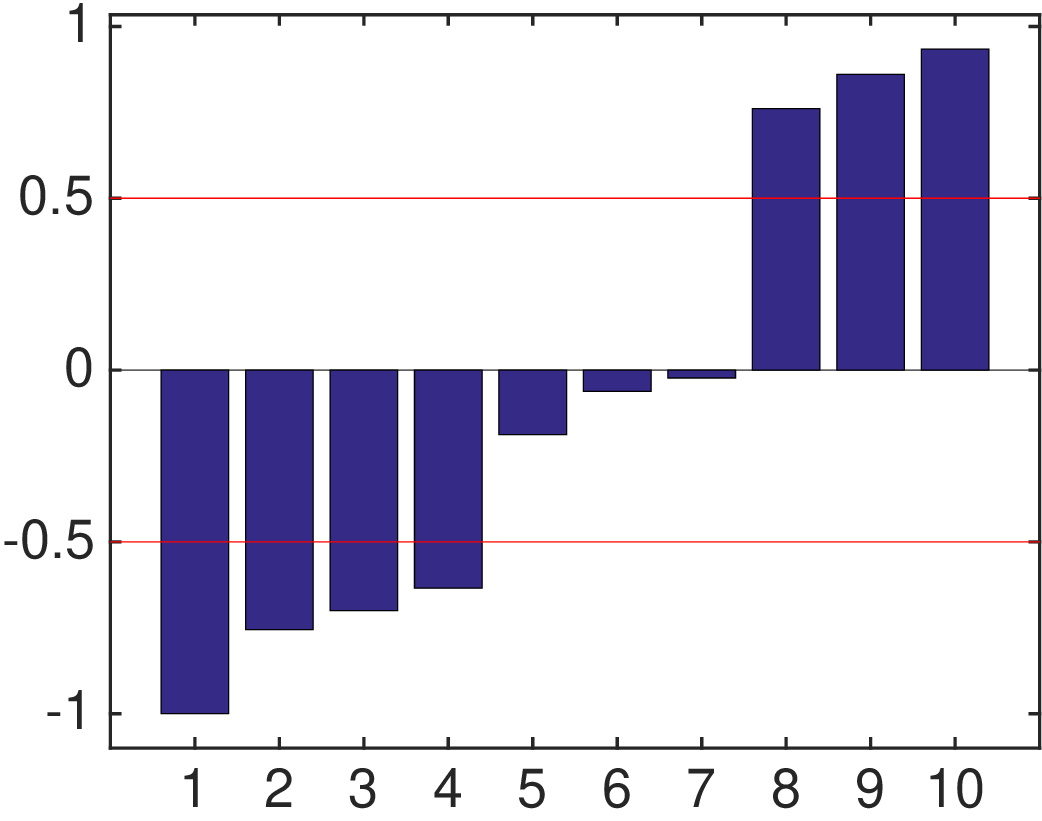}} \qquad
\subfloat[Convergence rate of the matrix Chebyshev expansion and decay rate of the Chebyshev coefficients]{\label{subfig:f3}  \includegraphics[width=0.4\textwidth]{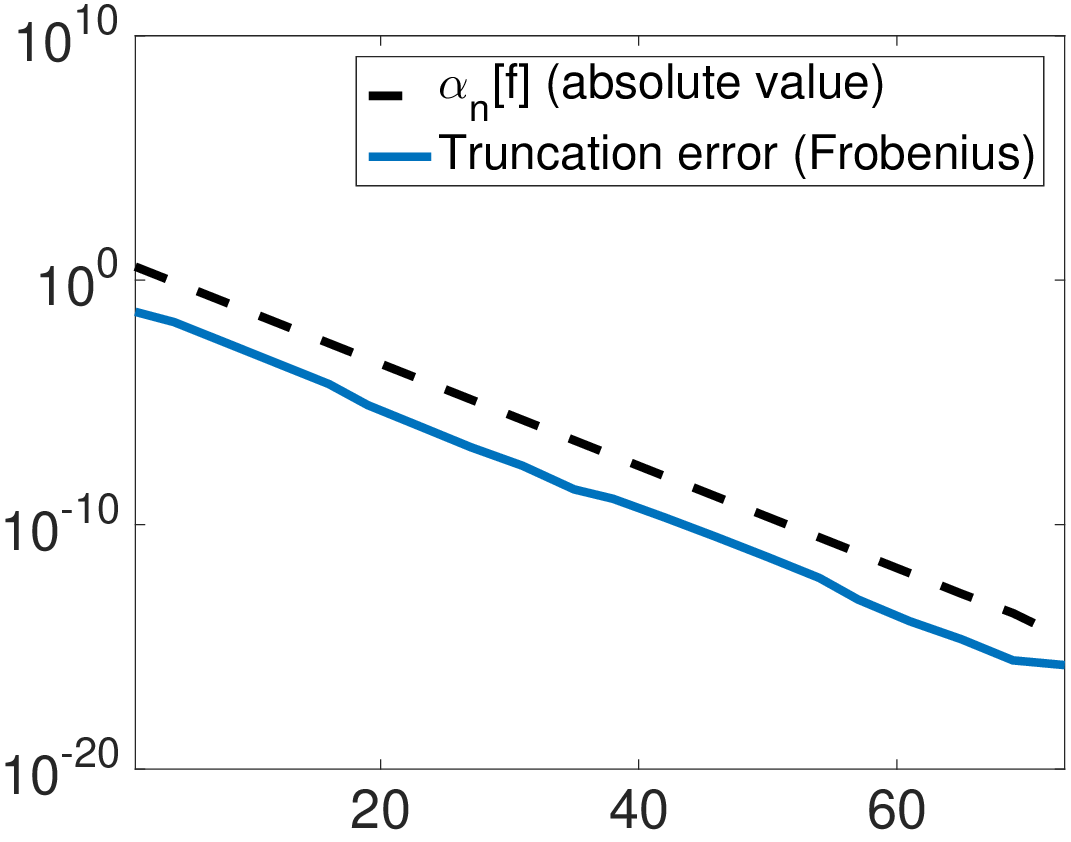}} 
 \caption{Evaluating $f_3$ of \eqref{eqn:f3_numerics} on a matrix of size $10 \times 10$.}
\end{figure}

\subsection{Matrix Chebyshev expansion on Jordan blocks}

As an example for applying matrix functions on non-diagonalizable matrices, we investigate the case of Jordan blocks. We focus on the two parameters of Jordan blocks: their eigenvalues and their size. 

We begin by considering the function
\begin{equation} \label{eqn:numeric_f4}
f_4(x) = \abs{x}^{3.5}  ,\quad x \in [-1,1] .
\end{equation}
This function has $3$ absolutely continuous derivatives, but the fourth one is not of bounded variation. Thus, for Jordan blocks with eigenvalue in $(-1,1)$, Corollary~\ref{cor:convergence_mid_segment} guarantees the convergence only for $m \le 3$. We evaluate $f_4$ for Jordan blocks associated with the eigenvalue $0.7$ and for sizes $m=2,3$ and $4$. The expansion of $f_4$ for the first two Jordan blocks is guaranteed to converge, while for the third it does not. The convergence rates of the three expansions are depicted  in Figure~\ref{subfig:Jordan1}, where the truncation error is given as a function of the expansion length. The error axis is on a logarithmic scale, and still there is a clear difference between the decay rates of the three cases. 

When considering $\delta$-condensed matrices, for $f_4$ and for Jordan blocks of size $3$ (with eigenvalue in $[-1+\delta,1-\delta]$, the convergence is guaranteed, while if the eigenvalues are $1$ or $-1$, the convergence is guaranteed only for Jordan blocks of size $2$. We plot in Figure~\ref{subfig:Jordan2} the truncation errors for three cases corresponding to the three eigenvalues $0.4$, $0.7$, and $1$. We see that the error associated with the eigenvalue $0.7$ is slightly higher than that of $0.4$, which can be explained by the growing constant \eqref{eqn:gamma_constant} and Lemma~\ref{lemma:leading_derivative_order}. Also, the error that is associated with the eigenvalue $1$ is large and its matrix Chebyshev expansion doesn't seem to converge.

\begin{figure}[ht]
    \centering
\subfloat[Three different sizes of Jordan blocks with eigenvalue $0.7$]{\label{subfig:Jordan1}  \includegraphics[width=0.4\textwidth]{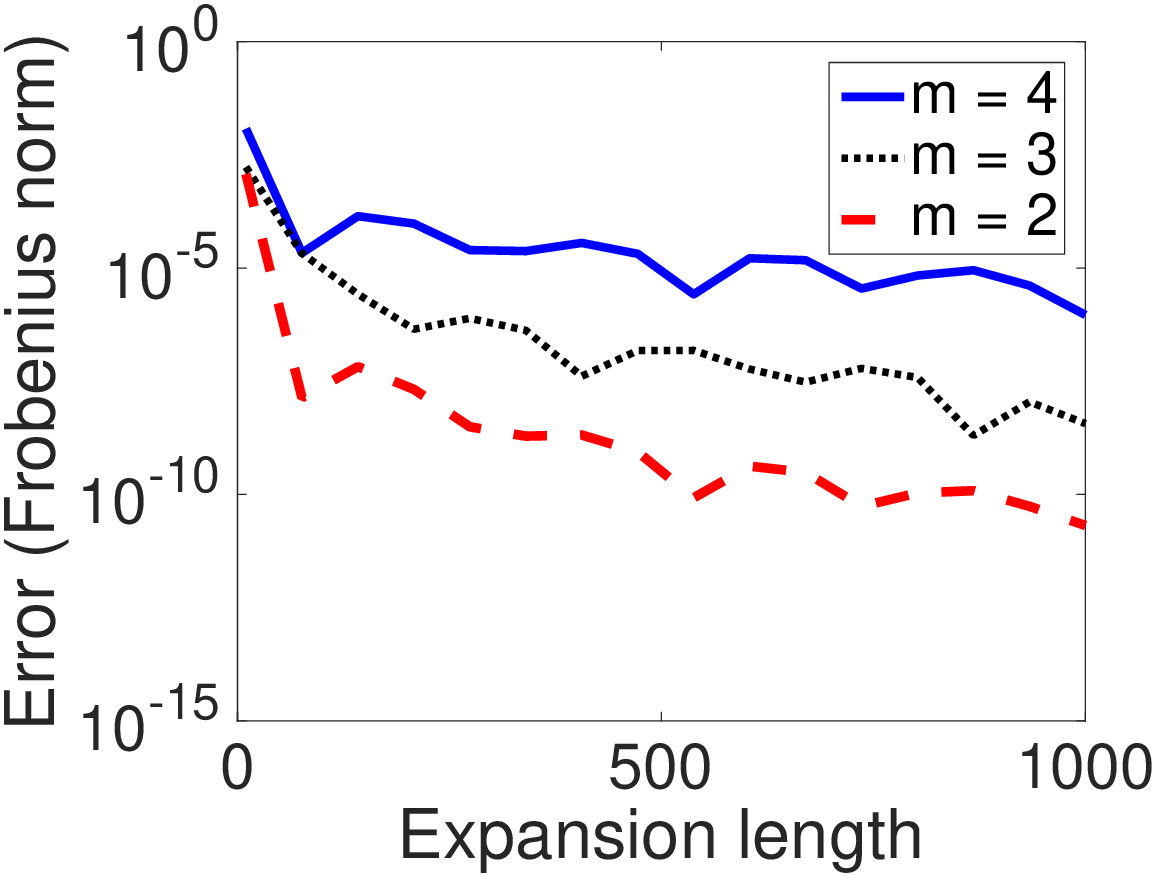}} \qquad
\subfloat[Three Jordan blocks of size $m=3$ corresponding to the eignevalues $0.4$, $0.7$, and $1$]{\label{subfig:Jordan2}  \includegraphics[width=0.41\textwidth]{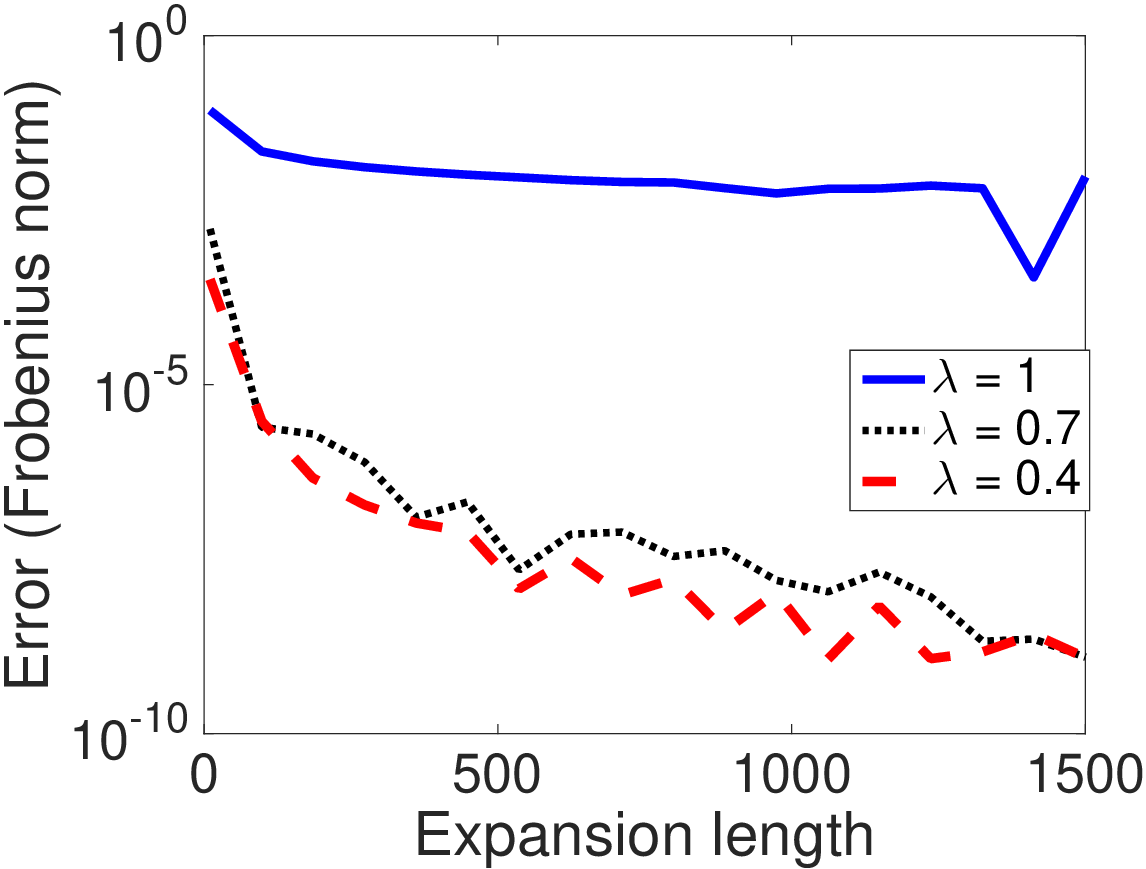}} 
 \caption{Truncation errors in evaluating $f_4$ of \eqref{eqn:numeric_f4} on Jordan blocks as a function of the expansion length.}
 \label{fig:differentJordanSize}
\end{figure}

The other parameter that affects the convergence rate is the size of the largest Jordan block. In particular, we show in Theorem~\ref{thm:convergence_rate} that while the size of the largest Jordan block affects the convergence rate, the matrix size does not. In order to test this observation numerically, we compare three block matrices of size $10 \times 10$, having blocks of varying sizes on their diagonal. These blocks have the eigenvalue $0.5$ on the diagonal and $0.5$ on super-diagonal entries (instead of the traditional $1$). This is done to avoid high condition numbers which can obscure the results of this test. The first matrix consists of a single block of size $10\times 10$, the second consists of two blocks of size $5 \times 5$ each, and the third consists of five blocks, each of size $2 \times 2$. We use the function $f_3$ of \eqref{eqn:f3_numerics} to allow for large Jordan blocks. The results are plotted in Figure \ref{fig:different_block_sizes}, where we notice that the convergence rates depend on the size of the largest Jordan block. It is worth noting that we repeated this test with a matrix having on the diagonal two copies of the above matrices, and obtained exactly identical results, which confirms numerically the observation of Theorem~\ref{thm:convergence_rate}. 

\begin{figure}
 \centering
 \includegraphics[width=.45\textwidth]{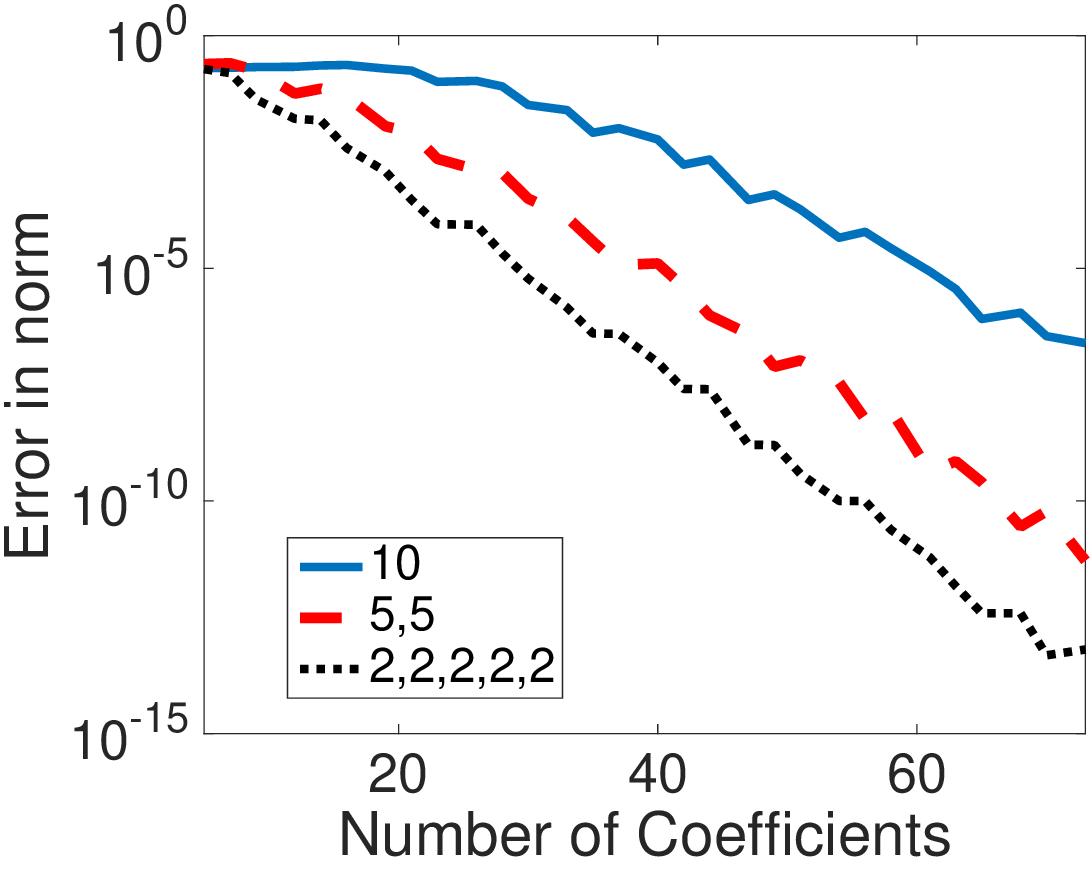}
 \caption{Convergence rates of $f_3$ from \eqref{eqn:f3_numerics} for three different block diagonal matrices. Each block corresponds to the eigenvalue $0.5$. In the legend are the different sizes of the blocks on the diagonal of our test matrix.}
 \label{fig:different_block_sizes}
\end{figure}


\subsection{Application to eigenspace recovery} \label{subsec:application}

We demonstrate the application of matrix Chebyshev expansions for estimating an eigenspace which corresponds to a given eigenvalue $\lambda$ of a matrix. We omit the classical case of the largest magnitude eigenvalue and consider the problem called the 'interior' eigenvalue problem. The standard solution for this problem is by the inverse power method, see \cite[Chapter 8]{golub2012matrix}, which applies the power method on $(A-\lambda I)^{-1}$. Applying the inverse power method requires solving at each iteration a system of linear equations of the same size as the matrix $A$. Thus, for large matrices, the conventional inverse power method might be computationally too expensive. Our demonstration here is a variant of a different approach, also known as Chebyshev acceleration \cite{saad1984chebyshev}. In this approach, instead of applying the power method (or one of its extension like the Lanczos algorithm) directly to $(A-\lambda I)^{-1}$, we apply it on a polynomial approximation of the inverse function. Nevertheless, the use of the inverse function $(x-\lambda)^{-1}$ is merely one example of a broader concept of filtering functions that increase the magnitude of a specific part of the spectrum while suppressing the rest.

The different variants of the above Chebyshev acceleration approach usually differ in the filter function they use. In \cite{fang2012filtered}, splines are used to form a filter function, and its polynomial approximation is applied to the matrix. A later approach \cite{li2016thick} directly uses a polynomial approximation to the Dirac function, based on least squares combined with a damping factor. This technique is derived from the Jackson kernel to avoid wigglings of the approximant. These variants aim to recover the eigenspaces associated with particular eigenvalues inside a segment (also termed 'spectrum slicing'). The ideal filter for such a task is the characteristic function over the required slice, but polynomial approximation of this ideal-filter is problematic due to its poor smoothness, as it is not even continuous.  

In our example, we use the filter function, 
\begin{equation} \label{eqn:filterFunc}
 f(x) =\frac{1}{2}\left( 1- \erf\left(  \frac{2}{r}( \abs{(x-c)}-R)   \right)  \right) , 
\end{equation}
centered around the eigenvalue whose eigenspace we want to recover. In~\eqref{eqn:filterFunc}, $\erf$ is the error function (integration over the Gaussian function), $R$ is a width parameter, and $r$ controls the steepness of the transition from zero to one (the smaller $r$ the steepest $f(x)$ is). The function $f(x)$ is smooth except at the center point $c$, where its first derivative has a jump discontinuity due to the absolute value. This discontinuity implies Gibbs phenomenon when approximating $f$ with the truncated Chebyshev expansion. Namely, there is always a fixed magnitude of approximation error centered around $c$. However, approximation errors around $c$ are negligible as the main objective is to suppress eigenvalues outside the region of interest. The parameter $R$ is roughly the distance to the nearby eigenvalue, if such an estimation is available. An illustration of \eqref{eqn:filterFunc} with different sets of parameters is given in Figure~\ref{fig:filteFunc}. 
\begin{figure}[ht]
 \centering
 \includegraphics[width=.45\textwidth]{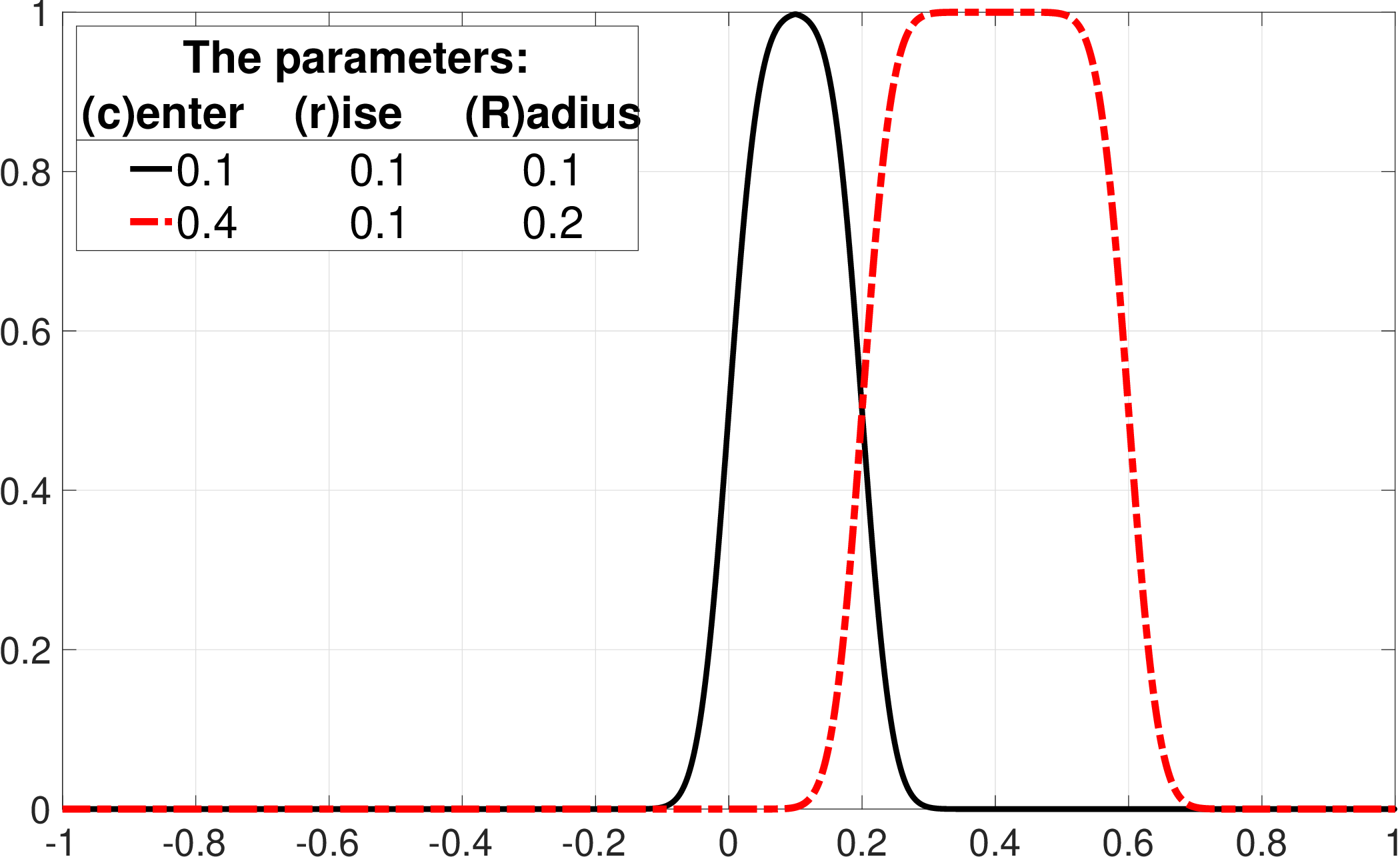}
  \caption{The filter function \eqref{eqn:filterFunc} for two different sets of parameters}
 \label{fig:filteFunc}
\end{figure}

We consider a scenario where the matrix size is large, and moreover we only have access to a function that applies the matrix to any a vector. In other words, we cannot access the entries of the matrix. Also, we a priori know an approximation of the eigenvalue whose eigenspace we wish to recover. To make this scenario realistic, we must assume there is some gap between this eigenvalue and the rest of the spectrum, for otherwise, any known variant of the power method would converge very slowly. We choose the parameters of the filter as follows: $c$ is the approximated eigenvalue, and $r$ and $R$ depend on the gap. As we decrease $r$ (corresponding to a smaller gap), the derivatives of $f$ increase, which means we have to use more coefficients in the truncated Chebyshev expansion to obtain the same prescribed accuracy of the approximation.

In the following, we present two examples. In the first example, we generate symmetric matrices of size up to $15,000 \times 15,000$, which is (approximately) the largest dense-matrix that fits in memory on our machine. Each matrix has a spectrum consisting of ones, zeros, and halves. The goal is to reveal the $20$ eigenvectors corresponding to the eigenvalue $0.5$. We compare two methods for this problem. The first is the inverse power method, implemented by solving a linear system of the form $(A-\lambda I)x=u$ with $\lambda=0.5$ at each iteration using Matlab's implementation (the \textit{EIGS} function with the backslash $A \backslash b$ solver). The second method is based on filtering via matrix Chebyshev expansion; we apply an evaluation of the matrix function $f(A)$, with the filter $f$ of~\eqref{eqn:filterFunc}, on an initial set of random vectors whose number is at least as large as the dimension of the subspace we wish to recover. To make a fair comparison between the two methods, we require the same precision of $10^{-10}$ when measuring the error by
$ \sum_{j=1}^k  \norm{Au_j-\lambda u_j} $.
To reach such a high precision using the matrix Chebyshev expansion, we need to truncate it at large $N$. However, that means computing many matrix-vector multiplications. Another way that we found quite efficient is to approximate $f(x)$ using some small $N$ such as $N=10$, but then to iterate the approximation several times. In particular, we reach the required precision in about seven such iterations. We introduce the results of the comparison between the two algorithms in Figure~\ref{fig:ComparisonEig}, where the gap between the two methods becomes larger as the size of the matrix increases.
\begin{figure}[ht]
 \centering
 \includegraphics[width=.45\textwidth]{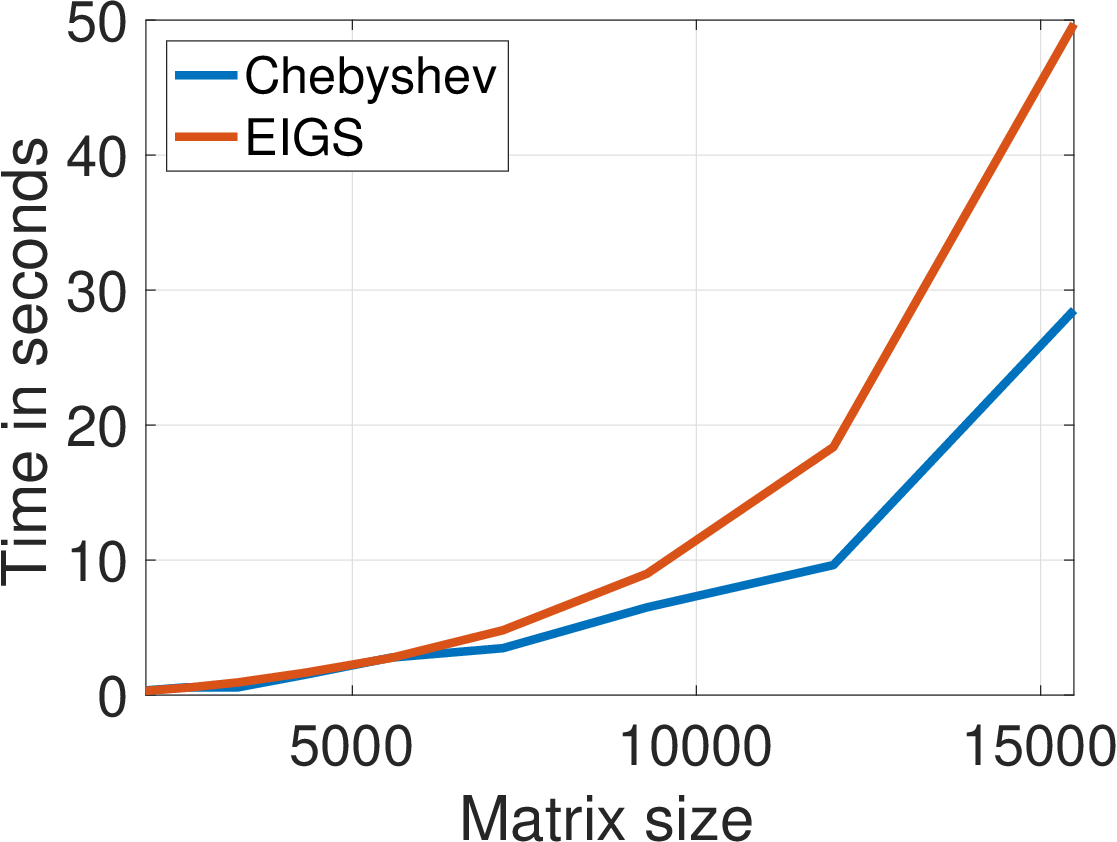}
  \caption{A timing comparison between iterative power method (by Matlab) and eigenspace recovery using matrix Chebyshev expansion. Time is measured in seconds, as a function of the matrix size.}
 \label{fig:ComparisonEig}
\end{figure}

In our second example, we use a particular class of matrices which are of the form $Q^TDQ$, where $D$ is a real diagonal matrix (the real spectrum), and $Q$ is the orthogonal matrix of the discrete cosine transform (DCT). Namely, for a vector $x$, $Qx$ is the discrete cosine transform of type $II$ of $x$. The advantage of using such a matrix for our example is that computing the vector products $Qx$ and $Q^{-1}y=Q^{T}y$ for any vectors $x$ and $y$ is extremely fast. We examine four sizes of matrices ranging from $50,000$ to $1,000,000$, each with a spectrum constructed as in the first example. Note that such large dense matrices cannot fit into the RAM of our machine. We measure the time required by our algorithm to estimate the required eigenspace to precision of $10^{-10}$, and report the results in Table~\ref{table:hugeMat}, indicating how fast and efficient it is to use the matrix Chebyshev expansion for this task. Note that solving linear systems of these sizes using conventional methods is far from being trivial and thus applying the iterative power method is intricate.

\begin{table}
\centering
\begin{tabular}{|c|c|c|c|c|} \hline
Matrix size     & 50,000 & 100,000 & 500,000 & 1,000,000  \\  \hline \hline
Timing (sec.) &   3.617 &   8.658 &  51.970 & 108.258  \\  \hline
Precision       & 2.758e-10 & 1.660e-11 & 3.677e-11 & 5.190e-11   \\  \hline
\end{tabular}
\caption{Timing (in seconds) of eigenspace recovery for big matrices.}
\label{table:hugeMat}
\end{table} 


\section*{Acknowledgments}
The authors would like to thank the anonymous reviewers for their valuable comments and suggestions.


\begin{appendices}
\label{appendix}


\section{Complementary proofs} \label{apx:proofsSection}

\subsection{Proof of Lemma \ref{lemma:Chebyshev_expansion_prop_from_fourier}} \label{sec:appendix_background_lemma}
\begin{proof}
By Theorem \ref{thm:Trefethen} we have that $\abs{\alpha_n[f]} \le \frac{C_f}{n^{2m}} $, with a positive constant $C_f$. By the bound \eqref{eqn:jordan_block_form},
we have that $\abs{\alpha_n[f] T_n^{(j)}(x)} \le M_n = C_f\frac{n^{2j}}{(2j-1)!n^{2m}} $, for any $ j=0,1,\ldots,m-1$, and thus
\begin{equation} \label{eqn:decay_derivative}
\sum_{n=1}^\infty M_n = \frac{C_f}{(2j-1)!}\sum_{n=1}^\infty  \frac{1}{n^{2(m-j)+2}} < \infty ,\quad j \le m-1 .
\end{equation}
Equation~\eqref{eqn:decay_derivative} guarantees that $\sumd{\infty} \abs{\alpha_n[f] T_n^{(j)}(x)} < \infty$, as required.

For the second part of Lemma~\ref{lemma:Chebyshev_expansion_prop_from_fourier}, we can use the Weierstrass M-test to conclude uniform convergence of
 \[ \sumd{\infty} \alpha_n[f] T^{(j)}_n(x) , \quad j=0,1,\ldots,m-1. \]
Moreover, it is clear that
\[ \sumd{N} \alpha_n[f] T^{(j)}_n(x) , \quad N \in \mathbb{N} , \quad j=0,1,\ldots,m-1 \]
 is a polynomial of degree $N$ and so obviously differentiable, and thus, we can differentiate it term-by-term.
\end{proof}

\subsection{The second leading order term of \texorpdfstring{$T^{(k)}_n(x)$}{Tkn(x)}  } \label{apx:second_LOT}

As noted in Remark \ref{rmk:LOT}, we can derive a bound on the constant in front of the $n^{k-1}$ term of the $k$-th derivative $T^{(k)}_n(x)$. We denote by $B_k$ this bound and prove by induction that 
\begin{equation} \label{eqn:boundOnBk}
B_k \le \frac{k(k-1)}{2} .
\end{equation}
For the base of the induction, we get from~\eqref{eqn:high_order_derivatives} that the constants are equal to $0$, $1$ and $3$, for $k=1,2$ and $3$, respectively. Assume that the claim is true for any $j$ that satisfies $0 \le j \le k <n $, for a fixed $k$. Then, for $k+1$, by \eqref{eqn:high_order_derivatives} we get the $n^k$ terms out of two factors: from the $n^k$ (leading order) term of $T^{(k)}_n$ that is multiplied by $(2k-1)\nu^2$, and from the $n^{k-2}$ term (second leading order) of $T^{(k-1)}_n$ that is multiplied by $\nu^2$. Combining the two factors leads to
\[ B_{k+1} \le (2k-1)\nu^{k+2} + \nu^2 B_{k-1} . \]
The latter recursion can be solved by repeatedly substituting \eqref{eqn:boundOnBk} to yield
\[ B_{k+1} \le \nu^{k+2}\sum_{j=1}^{\lfloor (k+1)/2 \rfloor} \left( 2(k-2j)+1 \right) = \frac{k(k+1)}{2}\nu^{k+2} , \]
as required for $k+1$.

\subsection{Proof of Theorem \ref{thm:convergence_rate}  } \label{apx:proofTruncError}

We organize the proof as follows. First, we define and briefly analyze an auxiliary Chebyshev term-by-term operator, acting from $\mathbb{R}$ to $\rms$. Next, we show how to adapt Theorem~\ref{thm:Trefethen} to the auxiliary operator. Finally, we show how to use this auxiliary operator for proving Theorem~\ref{thm:convergence_rate}.

We now define our auxiliary operator.
\begin{definition} \label{def:Chebyshev_term_by_term}
The term-by-term Chebyshev polynomial of order $N$ for a function $f \colon [-1,1] \to \rms$ is
\[  Q_N(f)(x) =  \sumd{N}\beta_n \left [f \right ] T_n(x)  , \]
where $\beta_n \left [f \right ] \in \rms$ with $(\beta_n \left [ f \right ])_{i,j} = \alpha_n \left [ f(x)_{i,j}\right ]$ for all $1 \le i,j \le k $.
\end{definition}
The coefficients $\{\beta_n\}_{n=0}^N$ of $Q_N$ are matrices defined by $\alpha_n$ (see~\eqref{eqn:chebyshev_series_scalars}) in each entry. We state the following relation between the Chebyshev term-by-term operator of Definition~\ref{def:Chebyshev_term_by_term} and the truncated Chebyshev expansion for scalar functions~\eqref{eqn:truncated_chebyshev}.
\begin{lemma}  \label{lemma:tr_ip_commute}
Let $f \colon [-1,1] \to \rms$ be a matrix-valued function such that $f(x)_{i,j} \in C([-1,1])$, for all $1 \le i,j \le k $. Then, for any fixed $B \in \rms$
\[
 \alpha_n \left [   \tr \left( f(t) B^T  \right) \right ] = \tr \left( \beta_n  \left[ f(t)  \right]  B^T \right) .
\]
In words, the trace and the (Chebyshev) inner product commute.
\end{lemma}
\begin{proof}
Define $h\colon [-1,1] \to \mathbb{R}$ by $h(t) = \tr \left( f(t) B^T  \right)$. The trace operator is linear and since $f$ is continuous in each entry so is $h$. Therefore, $\alpha_n\left [h\right ]$ is well-defined and
\begin{eqnarray*}
 \alpha_n \left [ h \right ] = \alpha_n \left[   \tr \left( f(t) B^T  \right) \right ] &=& \frac{2}{\pi} \int_{-1}^1 \frac{ \sum_{i,j=1}^k f(t)_{i,j} B_{i,j}  T_n(t)}{\sqrt{1-t^2}} dt \\
                  &=& \sum_{i,j=1}^k B_{i,j} \frac{2}{\pi} \int_{-1}^1 \frac{  f(t)_{i,j}   T_n(t)}{\sqrt{1-t^2}} dt \\
                  &=& \sum_{i,j=1}^k B_{i,j} \alpha_n \left [f(t)_{i,j}\right ] =  \tr \left( \beta_n  \left[ f(t)  \right]   B^T \right) .
\end{eqnarray*}
\end{proof}
Equipped with Lemma \ref{lemma:tr_ip_commute}, we are ready to lift Theorem \ref{thm:Trefethen} to the case of the Chebyshev term-by-term operator for matrix-valued functions. To this end, we use a result from the matrix duality theory: for any $X\in \rms$ and a matrix norm $\norm{\cdot}$, there exists a matrix $Y$ such that $\norm{X} =  \tr(XY^T)$ and $\norm{Y}^D = 1$, where $\norm{\cdot}^D$ is the dual norm defined as
\[ \norm{X}^D  = \max_{\norm{Y} \le 1} \{ \tr(XY^T)  \} .\]
The duality ensures that $(\norm{\cdot}^D)^D = \norm{\cdot}$ and that $\abs{\tr(XY^T)} \le \norm{X}\norm{Y}^D$. For more details see~\cite[Chapter 5]{horn2012matrix}). 
\begin{theorem}  \label{thm:error_term_by_term}
Let $f \colon [-1,1] \to \rms$ be such that $f_{i,j} \in C^{m-1}([-1,1])$ where $f_{i,j}^{(m-1)}$ is absolutely continuous and $f_{i,j}^{(m)}$ is of bounded variation for any $1 \le i,j \le k$. Then,
\[ \norm{f(x)-Q_N(f)(x)} \le C \frac{1}{(N-m)^m}  ,   \quad N>m,  \]
where $C$ is a constant independent of $x$ and $N$.
\end{theorem}
\begin{proof}
Denote by $E_N(Q,f)(x) =  f(x)-Q_N(f)(x)  $ the error matrix at $x$. By the duality theorem for matrix norms, for any given $A\in \rms$ and a matrix norm $\norm{\cdot}$, there exists a matrix $B$ with $\norm{B}^D = 1$ such that $\norm{A} =  \tr(AB^T)$. Now, let $x_0 \in [-1,1]$ be a fixed point and let $B_0$ be the corresponding matrix such that
\[ \norm{ E_N(Q,f)(x_0) }  = \tr \left( E_N(Q,f)(x_0) B_0^T \right) = \tr \left( f(x_0) B_0^T  \right) -  \tr \left( Q_N(f)(x_0) B_0^T  \right)  .\]
In other words, $\norm{E_N(Q,f)(x_0) } = h(x_0)  -  \tr \left(\sumd{N} \beta_n \left [f \right] T_n(x_0) B_0^T  \right)$ where $h(t) = \tr \left( f(t) B_0^T  \right) $ (as in the proof of Lemma \ref{lemma:tr_ip_commute}). Note that $T_n(x_0)$, $0 \le n \le N$, are scalars, and by the linearity of the trace operator
\begin{equation} \label{eqn:trEqualityA2}
 \tr \left(\sumd{N} \beta_n \left [f\right ] T_n(x_0) B_0^T  \right) =  \sumd{N} T_n(x_0) \tr \left( \beta_n \left [f \right ]  B_0^T  \right) .
\end{equation}
By Lemma \ref{lemma:tr_ip_commute}, the sum on the right-hand-side of \eqref{eqn:trEqualityA2} is equal to $\sumd{N} \alpha_n \left[   \tr \left( f(t) B_0^T  \right) \right] T_n(x_0)$, which means that
\begin{equation} \label{eqn:chebyshev_like_error_to_h}
 \norm{ E_N(Q,f)(x_0) } = h(x_0) - \sumd{N} \alpha_n [h] T_n(x_0) .
\end{equation}
The right hand side of \eqref{eqn:chebyshev_like_error_to_h} is the error in approximating $h$ by Chebyshev polynomials, and therefore, to bound it using Theorem~\ref{thm:Trefethen}, we need to verify that $h(x)$ satisfies the conditions of that theorem. The regularity of $h$ is directly inherited from the components of $f$, and so
\[ h^{(j)}(x) = \tr \left( f^{(j)}(x) B_0^T  \right) , \quad j=1,\ldots,m. \]
Thus, it remains to verify the finiteness of $\norm{h^{(m)}}_{TV}$.
By the duality theorem, $\abs{\tr(AB^T)} \le \norm{A}\norm{B}^D$, where in our case $\norm{B_0^T}^D=1$, and consequently,
\[ \norm{h^{(m)}}_{TV} \le \int_{-1}^1  \norm{f^{(m+1)}(t)}  \norm{B_0^T}^D dt \le  \int_{-1}^1 \norm{f^{(m+1)}(t) } dt  < \infty .\]
The last inequality can be easily verified using the max norm and the conditions on $f_{i,j}$. Applying Theorem~\ref{thm:Trefethen} to $h$ leads to the required bound due to \eqref{eqn:chebyshev_like_error_to_h}.
\end{proof}

There are two important remarks regarding Theorem \ref{thm:error_term_by_term}. First, the constant $C$ of the bound is independent of $k$ (the size of the matrix) and can be bounded by $\frac{2}{m} \int_{-1}^1 \norm{f^{(m+1)}(t) } dt$. Second, there are no restrictions on the matrix norm, except of being sub-multiplicative. These observations are also true for Theorem~\ref{thm:convergence_rate}.
\begin{proof}[Proof of Theorem~\ref{thm:convergence_rate}]
Let $g \colon [-1,1] \to \rms$ be a matrix-valued function defined by
\begin{equation}\label{eq:gx}
 g(x) = \sumd{\infty} \alpha_n[f] T_n(A) T_n(x).
\end{equation}
Note that Theorem \ref{thm:convergence_cheby} guarantees absolute convergence of the matrix Chebyshev expansion of $f$, and thus, $g$ is well-defined since
\[  \sumd{\infty} \norm{\alpha_n[f] T_n(A) T_n(x)} =\sumd{\infty}  \norm{\alpha_n[f] T_n(A)} \abs{T_n(x)} \le \sumd{\infty}  \norm{\alpha_n[f] T_n(A)} <\infty . \]
On the one hand, $ g(1) = f(A)$ since $T_n(1)=1$ for all $n =0,1,2,\ldots$. On the other hand, using the operator of Definition \ref{def:Chebyshev_term_by_term}
\begin{equation} \label{eqn:Qn_of_1}
 Q_N(g)(1) = \sumd{N} \beta_n \left [g \right ]T_n(1)  = \sumd{N} \beta_n \left [g \right ] ,
\end{equation}
where each matrix $\beta_n \left [g \right ]$ is defined as
\[ \beta_n \left [g \right ]_{i,j} =  \frac{2}{\pi} \int_{-1}^1 \frac{ \left( g(t)  \right)_{i,j} T_n(t) }{\sqrt{1-t^2}} dt
=  \frac{2}{\pi} \int_{-1}^1 \frac{ \left( \sum_{\ell=0}^\infty \alpha_\ell[f] T_\ell(A) T_\ell(t)   \right)_{i,j} T_n(t) } {\sqrt{1-t^2}} dt  . \]
Each element of $ \{ \sum_{\ell=0}^p \alpha_\ell[f] T_\ell(A) \}_{p=1}^\infty $ is bounded by $\sum_{\ell=0}^\infty \norm{\alpha_\ell[f] T_\ell(A)} $, and thus, we can use the bounded convergence theorem to interchange the integral and sum. Therefore, we have
\[ \beta_n \left [g \right ]_{i,j} =  \frac{2}{\pi} \sum_{\ell=0}^\infty \int_{-1}^1 \frac{ \left(  \alpha_\ell[f] T_\ell(A) T_\ell(t)   \right)_{i,j} T_n(t) } {\sqrt{1-t^2}} dt  . \]
Note that $T_\ell(t)$ is a scalar and $\alpha_\ell[f] T_\ell(A)$ is independent of the integration variable. Thus, we can rearrange the latter equation and use the orthogonality of the Chebyshev polynomials to get
\begin{eqnarray*}
 \beta_n \left [g \right ]_{i,j} &=&  \frac{2}{\pi} \sum_{\ell=0}^\infty \left(  \alpha_\ell[f] T_\ell(A)  \right)_{i,j} \int_{-1}^1 \frac{ T_\ell(t)  T_n(t) } {\sqrt{1-t^2}} dt \\
&=&  \sum_{\ell=0}^\infty \left(  \alpha_\ell[f] T_\ell(A)  \right)_{i,j} \delta_{\ell,n} = \left(  \alpha_n[f] T_n(A)  \right)_{i,j},
\end{eqnarray*}
where $\delta_{\ell,n}$ is the Kronecker delta. Combining the latter equation with \eqref{eqn:Qn_of_1} we get that $Q_N(g)(1) =  \sumd{N} \alpha_n[f] T_n(A)=S_N(f)(A)$, and therefore,
\begin{equation} \label{eqn:from_QN_to_SN}
 \norm{g(1) - Q_N(g)(1)} = \norm{f(A)-S_N(f)(A)} .
\end{equation}
It remains to show that the assumptions on $f$ imply that $g\in C^{\ell+1}$ so that we can apply Theorem \ref{thm:error_term_by_term} on $g$ to get the error bound $\frac{C}{(N-\ell)^\ell}$. This is done by considering the bound $\abs{T_n^{(j)}(x)} \le n^{2j}$, and the uniform bound \eqref{eqn:TnA_norm1_bound} on $\norm{T_n(A)}$ from which we get
\[  \sum_{n=1}^{\infty} \norm{\alpha_n[f] T_n(A) T^{(j)}_n(x)} < \infty , \quad  0 \le j \le \ell+1 .\]
Thus, the Weierstrass M-test implies that the convergence is absolute and uniform and so the function $g$, defined by~\eqref{eq:gx}, has the required smoothness.
\end{proof}

\section{Clenshaw's algorithm for matrices} \label{apx:algorithm}

\begin{algorithm}
\begin{algorithmic}[1]
\REQUIRE Matrix $A$ of order $k$ and coefficients $ \{  c_i \}_{i=0}^N$.
\ENSURE Truncated Chebyshev expansion $\sumd{N} c_{n}T_{n}\left(A\right)$.
\STATE $T\gets 2 \cdot A - I_{k}$
\STATE $d\gets 0$
\STATE $dd\gets 0$
\FOR{$n\gets N$  downto $2$}
\STATE $b \gets d $
\STATE $d \gets 2 \cdot T \cdot d - dd + c_n \cdot  I_{k}$ \label{alg:loop_main_line}
\STATE $dd\gets b$
\ENDFOR
\STATE \textbf{return} $T \cdot d - dd + 0.5 \cdot c_0 \cdot I_{k} $  \label{alg:final_line}
\end{algorithmic}
\caption{Clenshaw's algorithm for evaluating a truncated matrix Chebyshev expansion}
\label{alg:clenshaw}
\end{algorithm}


\end{appendices}


\bibliography{BibChebychevBound}
\bibliographystyle{plain}

\end{document}